\documentclass{amsart}
\usepackage[T1]{fontenc}
\usepackage[latin1]{inputenc}
\usepackage{graphicx}
\usepackage{amssymb}
\usepackage{amsmath}
\usepackage{amsthm}
\usepackage{color}
\usepackage[all]{xy}
\usepackage{amsfonts}
\usepackage{mathrsfs}
\usepackage{latexsym}
\usepackage{geometry} 
\usepackage{chngcntr}
\usepackage{textcomp}
\usepackage[english]{babel}

\theoremstyle{plain}
\newtheorem*{bigtheo}{Theorem}
\newtheorem{theo}{Theorem}[section]
\newtheorem{prop}[theo]{Proposition}
\newtheorem{lemm}[theo]{Lemma}
\newtheorem{coro}[theo]{Corollary}
\theoremstyle{definition}
\newtheorem{defi}[theo]{Definition}
\theoremstyle{remark}
\newtheorem{rema}[theo]{Remark}
\newtheorem{ex}[theo]{Example}
\newtheorem{hypo}[theo]{Hypothesis}

\counterwithin{table}{section}

\title{Partial Hasse invariants, partial degrees and the canonical subgroup}
\author{Stéphane Bijakowski}
\address{
Imperial College \\
Department of Mathematics \\
180 Queen's Gate \\
London SW7 2AZ UK}
\email{s.bijakowski@imperial.ac.uk}

\begin{document}
\selectlanguage{english}

\begin{abstract}
If the Hasse invariant of a $p$-divisible group is small enough, then one can construct a canonical subgroup inside its $p$-torsion. We prove that, assuming the existence of a subgroup of adequate height in the $p$-torsion with high degree, the expected properties of the canonical subgroup can be easily proved, especially the relation betwwen its degree and the Hasse invariant. When one considers a $p$-divisible group with an action of the ring of integers of a (possibly ramified) finite extension of $\mathbb{Q}_p$, then much more can be said. We define partial Hasse invariants (they are natural in the unramified case, and generalize a construction of Reduzzi and Xiao in the general case), as well as partial degrees. After studying these functions, we compute the partial degrees of the canonical subgroup.
\end{abstract}

\maketitle

\tableofcontents

\section*{Introduction}

Let $p$ be a prime number. Let $K$ be a finite extension of $\mathbb{Q}_p$ and $O_K$ its ring of integers. If $A$ is an abelian scheme over $O_K$, we say that $A$ is ordinary at $p$ if the $p$-divisible group $A[p^\infty]$ is an extension of a multiplicative $p$-divisible group and an étale one. If it is the case, then there is only one subgroup of $A[p]$ of height the dimension of $A$ which is multiplicative. It lifts the kernel of the Frobenius in the special fiber. \\
\indent When $A$ is close to being ordinary at $p$, then a similar result holds. The fact that $A$ is ordinary at $p$ is equivalent to the fact that the Hasse invariant of $A$ is zero (the Hasse invariant is an element in $[0,1]$). The theory of the canonical subgroup says that if the Hasse invariant of $A$ is small enough, then one can construct a canonical subgroup inside $A[p]$, which is equal to the multiplicative part of $A[p]$ when $A$ is ordinary. This construction has been done by Katz \cite{Katz} and Lubin \cite{Lu} for the elliptic curves, and by Abbes and Mokrane (\cite{A-M}) for general abelian schemes.  \\
\indent The problem actually makes sense for a general $p$-divisible group (not necessarily attached to an abelian scheme) : one can define the Hasse invariant for a $p$-divisible group, and try to construct a canonical subgroup when the Hasse invariant is small enough. This has been done by Tian (\cite{Ti}), using global methods and resolutions of $p$-divisible groups by abelian schemes. In \cite{Fa}, a purely local construction has been made explicit. The canonical subgroup has been a very active research topic, let us mention the contributions of Andreatta-Gasbarri (\cite{A-Ga}), Conrad (\cite{Co}), Goren-Kassaei (\cite{G-K}), Hattori (\cite{Ha}) and Scholze (\cite{Sch}). \\
\indent Once the canonical subgroup has been constructed, it is important to have some extra information for it. Fargues has defined in \cite{Fa_deg} the degree of a finite flat group scheme over $O_K$. The main result of \cite{Fa} is then the construction of a canonical subgroup $C \subset G[p]$, where $G$ is a $p$-divisible group whose Hasse invariant is strictly less than $1/2$. Moreover the height of $C$ is the dimension of $G$, and the degree of the dual of $C$ is equal to the Hasse invariant.  \\
\indent We prove that the canonical subgroup is in fact characterized by these properties. Indeed, one has the following result.

\begin{bigtheo}
Let $K$ be a complete valuated extension of $\mathbb{Q}_p$, and let $G$ be a $p$-divisible group over $O_K$. Let $C$ be a finite flat subgroup of $G[p]$ whose height is the dimension of $G$. Suppose that $\deg C^D < 1/2$, where $C^D$ is the dual of $C$. Then $C$ is uniquely determined by these properties. One has the relation
$$\deg C^D = ha(G) \text{,} $$
where $ha(G)$ is the Hasse invariant of $G$. Moreover, $C$ is the kernel of the Frobenius modulo $p^{1-ha(G)}$.
\end{bigtheo}

If one supposes the existence of a subgroup of the right height, and whose dual has small degree, then it can easily be proved that this subgroup is canonical in some sense. The proof of the theorem is relatively simple, and relies on the properties of the degree function, together with the description of Tate-Oort (\cite{T-O}) for finite flat group schemes of order $p$. Note that there is no assumption on $p$ in this theorem, unlike the result in \cite{Fa}. It is then very natural to define the canonical subgroup as being a subgroup with prescribed height, and whose dual has sufficiently small degree. \\
$ $\\
\indent A key feature for the canonical subgroup is the relation between its degree and the Hasse invariant of the $p$-divisible group. When one considers a $p$-divisible group with additional structures, then much more can be said. Let $F$ be a finite extension of $\mathbb{Q}_p$, $O_F$ its ring of integers and suppose that $G$ is a $p$-divisible group with an action of $O_F$. Then it is possible to define partial Hasse invariants for $G$, partial degrees for the subgroups of $G[p]$, and to relate all these elements for the canonical subgroup. We will describe these invariants and the relations in the case where $F$ is either unramified or totally ramified, the general case being of combination of these two cases. \\
$ $\\
\indent Let $F$ be a finite unramified extension of $\mathbb{Q}_p$ and let $K$ be a complete valuated extension of $\mathbb{Q}_p$ containing $F$. Let $\mathcal{T}$ be the set of embeddings of $F$ into $\overline{\mathbb{Q}_p}$, and let $G$ be a $p$-divisible group over $O_K$ with an action of $O_F$. We recall that the Hasse invariant of $G$ is defined as the valuation of the determinant of the map
$$V : \omega_{\overline{G}} \to \omega_{\overline{G}^{(p)}} \text{,} $$
where $\overline{G} := G \times_{O_K} O_K /p$, the superscript $(p)$ means a twist by the Frobenius and $V$ is the Verschiebung. Since $G$ has an action of $O_F$, the $O_K$ module $\omega_G$ admits a direct sum decomposition according to the elements of $\mathcal{T}$ : 
$$\omega_G = \bigoplus_{\tau \in \mathcal{T}} \omega_{G,\tau} \text{.} $$
The map $V$ induces maps 
$$V_\tau : \omega_{\overline{G},\tau} \to \omega_{\overline{G}^{(p)},\sigma^{-1}\tau} \text{,} $$
where $\sigma \in \mathcal{T}$ is the Frobenius. One can then define partial Hasse invariants $ha_\tau (G) \in [0,1]$ as the valuation of the determinant of $V_\tau$ for all $\tau \in \mathcal{T}$. The sum of the partial Hasse invariants is the Hasse invariant $ha(G)$. \\
\indent If $H$ is a $O_F$-stable finite flat subgroup of $G[p]$, then one can define partial degrees $(\deg_\tau H)_\tau$ for $H$, as well as for its dual $H^D$. The sum of the partial degrees is the total degree. One has the following information concerning the canonical subgroup in that case.

\begin{bigtheo}
Let $F$ be a finite unramified extension of $\mathbb{Q}_p$ and let $K$ be a complete valuated extension of $\mathbb{Q}_p$ containing $F$. Let $G$ be a $p$-divisible group over $O_K$ with an action of $O_F$. Suppose that there exists a canonical subgroup $C$ for $G[p]$. Then one has 
$$\deg_\tau C^D = ha_\tau(G) \text{,} $$
for all $\tau \in \mathcal{T}$. If moreover the Hasse invariant of $G$ is strictly less than $1/(p+1)$ then 
$$ha_\tau (G/C) = p \cdot ha_{\sigma^{-1} \tau}(G) \text{,} $$
for all $\tau \in \mathcal{T}$.
\end{bigtheo}

Note that the computations of the partial Hasse invariants $ha_\tau (G/C)$ were already done in \cite{G-K} for the Hilbert modular variety. \\
\indent The definition of the partial Hasse invariants and the partial degrees is very natural when $F$ is unramified. The situation is more involved in the ramified case. Suppose now that $F$ is a totally ramified extension of $\mathbb{Q}_p$ of degree $e \geq 2$ with uniformizer $\pi$. Let $K$ be a complete valuated extension of $\mathbb{Q}_p$ containing the Galois closure of $F$ and let $G$ be a $p$-divisible group over $O_K$ with an action of $O_F$. The $O_K$-module $\omega_G$ do not split under the action of $O_F$, but one has a filtration
$$0 \subset \omega_{G}^{[1]} \subset \dots \subset \omega_G^{[e]}=\omega_G \text{,} $$
where $\omega_{G}^{[j]}/ \omega_G^{[j-1]}$ is free over $O_K$ and with $O_F$ acting on it by a fixed embedding. This filtration is well defined once we have fixed an ordering on $\Sigma$, the set of embeddings of $F$ into $\overline{\mathbb{Q}_p}$. The construction of the partial Hasse invariants for the special fiber of the Hilbert modular variety has been done by Reduzzi and Xiao (\cite{R-X}), and the generalization of their method is straightforward. Let us describe briefly this construction. The Verschiebung map respects the filtration on $\omega_G \otimes_{O_K}~O_K / \pi$; the valuations of the determinants of $V$ acting on the graded pieces give elements $ha^{[1]}(G), \dots, ha^{[e]}(G)$ in $[0,1/e]$ that we call the partial Hasse invariants. Moreover, one can decompose each of these invariants. The action of $O_F$ gives a map $[\pi] : \omega_G \to \omega_G$. If we denote by $\omega_{G,\{1/e\}} := \omega_G \otimes O_K / \pi O_K$, and similarly for $\omega_{G,\{1/e\}}^{[j]}$, then the maps $[\pi]$ sends $\omega_{G,\{1/e\}}^{[j]}$ into $\omega_{G,\{1/e\}}^{[j-1]}$ for all $1 \leq j \leq e$. One then gets a map
$$\omega_{G,\{1/e\}}^{[j]} / \omega_{G,\{1/e\}}^{[j-1]} \to \omega_{G,\{1/e\}}^{[j-1]} / \omega_{G,\{1/e\}}^{[j-2]} \text{,} $$
for all $2 \leq j \leq e$. The valuation of the determinant of this map will be denoted by $m^{[j]}(G)$. One also gets a map
$$\omega_{G,\{1/e\}}^{[1]} \to (\omega_{G,\{1/e\}}/ \omega_{G,\{1/e\}}^{[e-1]})^{(p)} \text{,} $$
where the superscript $(p)$ means a twist by a Frobenius. This map can be thought as the composition of the division by $[\pi]^{e-1}$ and the Verschiebung map. The valuation of the determinant of this map will be written $hasse(G)$. \\
\indent One can then define primitive Hasse invariants $(hasse(G), m^{[2]}(G), \dots, m^{[e]}(G))$; the partial invariants $ha^{[j]}(G)$ can moreover be expressed as linear combinations of these primitive Hasse invariants. Note that the element $hasse(G)$ is related to the Verschiebung, whereas the elements $m^{[j]}(G)$ depend only on the structure of $\omega_G$ as an $O_K \otimes_{\mathbb{Z}_p} O_F$-module. The relations $m^{[j]}(G)=0$ for all $2 \leq j \leq e$ are equivalent to the fact that $\omega_G$ is free over $O_K \otimes_{\mathbb{Z}_p} O_F$ (this is usually called the Rapoport condition). We prove a duality result for these partial and primitive Hasse invariants (section \ref{dual}). We also show that they do not depend on the choice of any ordering for $\Sigma$ if the total Hasse invariant is strictly less than $1/e$ (see Proposition \ref{final_indep}). \\
\indent If $H$ is an $O_F$-stable finite flat subgroup of $G[p]$, then one can define partial degrees $\deg^{[j]} H$. Indeed, one has a map
$$\omega_{G/H} \to \omega_G \text{,} $$
and the valuation of the determinant of this map is the degree of $H$. This map respects the filtration on each of the two modules, so one gets maps
$$\omega_{G/H}^{[j]} / \omega_{G/H}^{[j-1]} \to \omega_{G}^{[j]} / \omega_G^{[j-1]}$$ 
for all $1 \leq j \leq e$. The valuation of the determinant of this map is by the definition the partial degree of $H$. Considering the map $\omega_{G^D} \to \omega_{(G/H)^D}$, one defines similarly the partial degrees of the dual of $H$. We prove some properties for these partial degrees (additivity, compatibility with duality), and we also prove that if the degree of $H$ (or its dual) is sufficiently small, then the partial degrees do not depend on any choice for the set $\Sigma$ (see section \ref{deg_prop}).  \\
\indent In this setting, we prove the following properties for the canonical subgroup.
 
\begin{bigtheo}
Let $F$ be a totally ramified extension of $\mathbb{Q}_p$ of degree $e \geq 2$ with uniformizer $\pi$, and let $K$ be a complete valuated extension of $\mathbb{Q}_p$ containing the Galois closure of $F$. Let $G$ be a $p$-divisible group over $O_K$ with an action of $O_F$. Suppose that there exists a canonical subgroup $C \subset G[p]$, and suppose that the Hasse invariant of $G$ is strictly less than $1/e$. Then 
$$\deg (C[\pi^k] / C[\pi^{k-1}])^D = ha^{[e-k+1]}(G)$$
for all $1 \leq k \leq e$. Moreover, one has
$$\deg^{[1]} C[\pi]^D = hasse(G) \text{ and } \deg^{[j]} C[\pi]^D = m^{[j]}(G)$$
for all $2 \leq j \leq e$. If $ha(G) < 1/(pe)$ and if there is a canonical subgroup for $G/C$, then one has
$$ha^{[1]}(G/C[\pi]) = p \cdot ha^{[e]}(G) \text{ and } ha^{[j]}(G/C[\pi]) = ha^{[j-1]} (G)$$
for all $2 \leq j \leq e$.
\end{bigtheo}

One can then relate the degree of the groups $C[\pi^k] / C[\pi^{k-1}]$ to the partial Hasse invariants, and the partial degrees of $C[\pi]$ to the primitive Hasse invariants. One can also compute the partial Hasse invariants of $G/C[\pi]$. Actually, one can have more relations, and compute the partial degrees of $C[\pi^k]/C[\pi^{k-1}]$, and the partial and primitive Hasse invariants of $G/C[\pi^k]$ for all $1 \leq k \leq e$ (see tables $\ref{tab}$ and $\ref{tabl}$). \\
$ $\\
\indent Let us now talk about the organization of the paper. In the first part, we define the partial and primitive Hasse invariants for a $p$-divisible group with an action, and prove certain properties for these invariants. In the second part, we define and study the partial degrees for a finite flat subgroup of such a $p$-divisible group. In the third section, we first describe an alternative approach to the canonical subgroup. Then we consider the canonical subgroup of a $p$-divisible group with an action, and relate its partial degrees to the primitive and partial Hasse invariants. \\
$ $\\
\indent The author would like to thank Valentin Hernandez for helpful discussions.

\section*{Notations}

Let $F$ be a finite extension of $\mathbb{Q}_p$. Let $f$ and $e$ be respectively the residual degree and the ramification index, and let $O_F$ denote the ring of integer of $F$. We will write $F^{ur}$ for the maximal unramified extension of $\mathbb{Q}_p$ contained in $F$ and $O_{F^{ur}}$ for its ring of integer; it is an extension of $\mathbb{Q}_p$ of degree $f$. Let $\pi$ be an uniformizer of $F$. \\
\indent Let $\mathcal{T}$ and $\Sigma$ be respectively the set of embeddings of $F^{ur}$ and $F$ into $\overline{\mathbb{Q}_p}$. For each $\tau \in \mathcal{T}$, we denote by $\Sigma_\tau$ the set of $\sigma \in \Sigma$, such that $\sigma$ is equal to $\tau$ by restriction to $F^{ur}$. We will write $\mathcal{T} = \{\tau_1, \dots, \tau_f\}$, such that $\tau_{i+1}=\sigma \tau_i$ for $1 \leq i \leq f-1$, where $\sigma$ is the Frobenius. We thus have an identification between $\mathcal{T}$ and $\{1, \dots, f\}$. \\
\indent Let $K$ be a complete valuated field which is an extension of $\mathbb{Q}_p$. We suppose that $K$ contains the Galois closure of $F$. We normalize the valuation of $K$ such that $v(p)=1$. Let $O_K$ be the valuation ring of $K$, and $k$ the residue field. If $M$ is an $O_K$-module with an action of $O_{F^{ur}}$, then there is a decomposition
$$M = \bigoplus_{i=1}^f M_i \text{,} $$
where $M_i$ consists of the elements of $M$ where $O_{F^{ur}}$ acts by $\tau_i$. \\
\indent For all $\alpha >0$ we will write $\mathfrak{m}_\alpha := \{x \in O_K, v(x) \geq \alpha \}$. If $M$ is a $O_K$-module, we write $M_{\{\alpha\}} : =M  \otimes_{O_K} O_K / \mathfrak{m}_\alpha$. If $M$ is a free $O_{K,\{\alpha\}}$-module of finite rank, with $\alpha \leq 1$, we define $M^{(p)} := M \otimes_{O_{K,\{\alpha\}},\varphi} O_{K,\{\alpha\}}$, where $\varphi$ is the Frobenius acting on $O_{K,\{\alpha\}}$. This is still a free $O_{K,\{\alpha\}}$ module of the same rank. \\
$ $\\
\indent In this paper, we will consider a $p$-divisible group $G$ defined over $O_K$ endowed with an action of $O_F$. In other words, we suppose the existence of a morphism
$$O_F \to \text{End}(G) \text{.} $$
The height of $G$ is thus divisible by $ef$, we will denote by $h$ this height divided by $ef$. Let $\omega_G$ be the conormal sheaf of $G$ along its unit section; it is a free $O_K$-module of rank the dimension of $G$ which has an action of $O_F$. We will make the following hypothesis throughout this article.

\begin{hypo} \label{hypothese}
The $K$-vector space $\omega_G \otimes_{O_K} K$ is a free $K \otimes_{\mathbb{Z}_p} O_F$-module. 
\end{hypo}

This condition says that there is no obstruction for $G$ to be ordinary. In general, there exists a decomposition $\omega_G \otimes_{O_K} K = \oplus_{\sigma \in \Sigma} V_\sigma$, with $O_F$ acting on $V_\sigma$ by $\sigma$. The hypothesis is then equivalent to the fact that the dimension of $V_\sigma$ is independent of $\sigma$. Let $d$ be the dimension of any $V_\sigma$; the dimension of $G$ is then $efd$. If $\omega_G$ is a free $O_K \otimes_{\mathbb{Z}_p} O_F$-module, then we say that $G$ satisfies the Rapoport condition. \\
$ $\\
\indent The module $\omega_G$ has an action of $O_{F^{ur}}$, thus has a decomposition
$$\omega_G = \bigoplus_{i=1}^f \omega_{G,i} \text{,} $$
where $\omega_{G,i}$ is a free $O_K$-module of rank $ed$, with $O_{F^{ur}}$ acting on it by $\tau_i$. To simplify the notations, we will just write $\omega$ and $\omega_i$ for $\omega_G$ and $\omega_{G,i}$ when there is no possible confusion.

\section{Partial Hasse invariants}

\subsection{Definition of the invariants}

Let $\overline{G}$ denote the reduction of $G$ to $O_K/p$, and let $\overline{G}^{(p)}$ be the twist of $G$ by the Frobenius. We have the Verschiebung map
$$V : \omega_{\overline{G}} \to \omega_{\overline{G}^{(p)}} \text{.} $$
But we have $\omega_{\overline{G}} = \omega_{\{1\}} $ and $\omega_{\overline{G}^{(p)}} = (\omega_{\{1\}})^{(p)}$. We thus get a map
$$V : \omega_{\{1\}} \to (\omega_{\{1\}})^{(p)} \text{.} $$

The module $\omega_{\{1\}}$ is free of rank $efd$ over $O_{K,\{1\}}$. By fixing bases and taking the determinant of the previous map, one gets an element that we will denote by $Ha(G)$. It is an element of $O_{K,\{1\}}$. Its truncated valuation is the Hasse invariant and will be denote by $ha(G) \in [0,1]$. Since $G$ has an action of $O_F$, one can refine this invariant, and define partial Hasse invariants. They are natural in the unramified case. \\
\indent Recall that we have a decomposition $\omega = \oplus_{i=1}^f \omega_i$. The Verschiebung map induce maps
$$V_i : \omega_{i,\{1\}} \to (\omega_{i-1,\{1\}})^{(p)}$$
for all $1 \leq i \leq f$ (here and later, we set $\omega_{0} = \omega_{f}$). Each module $\omega_{i,\{1\}}$ is free of rank $ed$ over $O_{K,\{1\}}$.

\begin{defi}
The element ${Ha}_i(G)$ is defined as the determinant of $V_i$. It is an element of $O_{K,\{1\}}$. Its truncated valuation will be denoted by $ha_i(G) \in [0,1]$.
\end{defi}

We call the elements $(ha_i(G))_i$ the \textit{unramified partial Hasse invariants}. If $F$ is unramified over $\mathbb{Q}_p$, we have thus constructed all the partial Hasse invariants claimed in the introduction. The situation is more involved in the ramified case. Their definition is a straightforward generalization of a construction from Reduzzi and Xiao for the special fiber of the Hilbert modular variety (see \cite{R-X}). \\
$ $\\
\indent Let us fix an element $1 \leq i \leq f$ and consider the free $O_K$-module $\omega_{i}$. It has an action of $O_F$, and $O_{F^{ur}}$ acts on it by $\tau_i$. Let us write $\Sigma_i$ for $\Sigma_{\tau_i}$; we recall that it consists of the elements of $\Sigma$ which are equal to $\tau_i$ by restriction to $F^{ur}$. Let us fix an ordering on this set : $\Sigma_i = \{\sigma_{i,1}, \dots, \sigma_{i,e} \}$. The $K$-vector space $\omega_{i} \otimes_{O_K} K$ has a natural decomposition
$$\omega_{i} \otimes_{O_K} K = \bigoplus_{j=1}^e N_{i,j} \text{,} $$
where $N_{i,j}$ consists of the elements of $\omega_i \otimes_{O_K} K$ with $O_F$ acting on them by $\sigma_{i,j}$. This gives a filtration on $\omega_G \otimes_{O_K} K$, by considering the subspaces $F_{i,j} =\oplus_{k=1}^j N_{i,k}$. We can pull back this filtration to $\omega_{i}$, and get a filtration
$$0=\omega_{i}^{[0]} \subset \omega_{i}^{[1]} \subset \dots \subset \omega_{i}^{[e]}=\omega_{i} \text{.} $$
Each $\omega_{i}^{[j]}$ is a free $O_K$-module of rank $dj$, and we have $\omega_{i}^{[j]} \otimes_{O_K} K = F_{i,j}$. By consequence, each of the graded pieces $\omega_{i}^{[j]} / \omega_{i}^{[j-1]}$ is a free $O_K$-module of rank $d$, is isomorphic over $K$ to $N_{i,j}$ and $O_F$ acts by $\sigma_{i,j}$ on it for all $1 \leq j \leq e$. \\
\indent The uniformizer $\pi$ of $F$ acts on $\omega_i$; we will denote $ [\pi] : \omega_i \to \omega_i$ the map induced by its action. This map acts on $\omega_{i}^{[j]} / \omega_{i}^{[j-1]}$ as the scalar $\sigma_{i,j} (\pi)$ for all $1 \leq j \leq e$. This element has valuation~$1/e$; thus if we reduce modulo $\mathfrak{m}_{1/e}$, the map $[\pi]$ will be trivial on the graded pieces. More precisely, for all $1 \leq j \leq e$, we have a map
$$[\pi] : \omega_{i,\{1/e\}}^{[j]} \to  \omega_{i,\{1/e\}}^{[j-1]} \text{.} $$

\begin{defi}
For all $1 \leq i \leq f$, and for all $2 \leq j \leq e$, we write $M_{i}^{[j]}$ the map
$$[\pi] : \omega_{i,\{1/e\}}^{[j]} / \omega_{i,\{1/e\}}^{[j-1]} \to \omega_{i,\{1/e\}}^{[j-1]} / \omega_{i,\{1/e\}}^{[j-2]} \text{.} $$
We write $M_i^{[j]} (G)$ the determinant of this map; it is an element of $O_{K,\{1/e\}}$. We also define $m_i^{[j]}(G) :=~v(M_i^{[j]}(G)) \in [0,1/e]$.
\end{defi}

Note that all the graded parts $\omega_{i,\{1/e\}}^{[j]} / \omega_{i,\{1/e\}}^{[j-1]}$ are free of rank $d$ over $O_{K,\{1/e\}}$. The element $M_i^{[j]}(G)$ depends on the choice of the uniformizer $\pi$, but its valuation $m_i^{[j]}(G)$ does not. These elements also depend on the choice of an ordering for the set $\Sigma_i$. To study this dependence, we first make a definition.

\begin{defi}
A filtration $0 = F_{i,\{1/e\}}^{[0]} \subset F_{i,\{1/e\}}^{[1]} \subset \dots \subset F_{i,\{1/e\}}^{[e]}=\omega_{i,\{1/e\}}$ is called \textit{adequate} if the following conditions are satisfied. 
\begin{itemize}
\item Each $F_{i,\{1/e\}}^{[j]}/F_{i,\{1/e\}}^{[j-1]}$ is a free $O_{K,\{1/e\}}$-module of rank $d$, for $1 \leq j \leq e$.
\item The map $[\pi]$ sends $F_{i,\{1/e\}}^{[j]}$ into $F_{i,\{1/e\}}^{[j-1]}$ for $1 \leq j \leq e$.
\end{itemize}
\end{defi}

The filtration $(\omega_{i,\{1/e\}}^{[j]})$ we constructed is thus adequate. Each adequate filtration gives maps
$$[\pi] : F_{i,\{1/e\}}^{[j]} / F_{i,\{1/e\}}^{[j-1]} \to F_{i,\{1/e\}}^{[j-1]} / F_{i,\{1/e\}}^{[j-2]} \text{,} $$
for $2 \leq j \leq e$ and thus elements $(m_i^{[j]})$, with $m_i^{[j]} \in [0,1/e]$. If these elements are small enough, then they do not depend on the adequate filtration. Indeed, we have the following propositions.

\begin{prop} \label{independance}
Let $(F_{i,\{1/e\}}^{[j]})$ be an adequate filtration of $\omega_{i,\{1/e\}}$, with invariants $m_i^{[j]}$. Let $r_i = \sum_{k=2}^e (k-1) m_i^{[k]}$ and suppose $r_i < 1/e$. If $({F_{i,\{1/e\}}^{[j]}}')$ is another adequate filtration, then we have
$$ F_{i,\{1/e-r_i\}}^{[j]} = {F_{i,\{1/e-r_i\}}^{[j]}}'$$
for all $1 \leq j \leq e$.
\end{prop}

\begin{proof}
We fix a basis of $\omega_{i,\{1/e\}}$ adapted to the filtration $(F_{i,\{1/e\}}^{[j]})$. The map $[\pi]$ acts on $\omega_{i,\{1/e\}}$ by a matrix of the form
\begin{displaymath}
\left(
\begin{array}{ccccc}
0 & M_i^{[2]}  & \dots & * & * \\
 & 0  & \ddots  & * &* \\
& &  \ddots & M_i^{[e-1]} & N_{e-1,e} \\
& & & 0 & M_i^{[e]} \\
& & & & 0 
\end{array}
\right)  \text{.} 
\end{displaymath}
The elements $x \in {F_{i,\{1/e\}}^{[1]}}'$ verify $[\pi] x = 0$.  Let us write the coordinates of $x$ in the previous basis by
$$x = \left( \begin{array}{c}
X_1 \\
\vdots \\
X_e
\end{array}
 \right) \text{.} $$
The relation $[\pi] x = 0$ gives $M_i^{[e]} X_e = 0$. Since the determinant of $M_i^{[e]}$ has valuation $m_i^{[e]}$, the vector $X_e$ has coefficients in $\mathfrak{m}_{1/e - m_i^{[e]}}$. We thus have $X_e=0$ in $\omega_{i,\{1/e - m_i^{[e]}\}}$. We also have the relation $M_i^{[e-1]} X_{e-1} + N_{e-1,e} X_e = 0$. In $\omega_{i,\{1/e - m_i^{[e]}\}}$, we then have $M_i^{[e-1]} X_{e-1} = 0$. Thus $X_{e-1} = 0$ in $\omega_{i,\{1/e - m_i^{[e]} - m_i^{[e-1]}\}}$. Let us write $r_i^{[j]} = \sum_{k=j}^e m_i^{[k]}$ for all $2 \leq j \leq e$. Finally, we see that $x \in F_{i,\{1/e - r_i^{[2]}\}}^{[1]}$, so that
$${F_{i,\{1/e- r_i^{[2]}\}}^{[1]}}' = F_{i,\{1/e- r_i^{[2]}\}}^{[1]} \text{.} $$
We can then work by induction, considering $\omega_{i,\{1/e- r_i^{[2]} \}} / F_{i,\{1/e- r_i^{[2]}\}}^{[1]}$. We then get
$${F_{\{1/e- r_i^{[2]} - \dots - r_i^{[j+1]} \}}^{[j]}}' = F_{\{1/e- r_i^{[2]} - \dots - r_i^{[j+1]}\}}^{[j]}$$
for all $1 \leq j \leq e-1$. Hence the result since $r_i^{[2]} + \dots + r_i^{[e]} = r_i$.
\end{proof}

\noindent We will write $r_i(G) = \sum_{k=2}^e (k-1) m_i^{[k]} (G)$ for all $1 \leq i \leq f$.

\begin{coro} \label{indep_m}
Let $i$ be an integer between $1$ and $f$, and suppose that we have
$$m_{i}^{[j]}(G) + r_i(G) < 1/e$$
for all $1 \leq j \leq e$. Then the elements $(m_i^{[j]} (G))_{1 \leq j \leq e}$ can be computed using any adequate filtration on $\omega_{i,\{1/e\}}$. In particular, they do not depend on an ordering for the set $\Sigma_i$.
\end{coro}

\begin{proof}
Let $(F_{i,\{1/e\}}^{[j]})$ be an adequate filtration; from the previous proposition we get
$$F_{i,\{1/e - r_i(G)\}}^{[j]} = \omega_{i,\{1/e - r_i(G)\}}^{[j]}$$
for all $1 \leq j \leq e$. The map $[\pi] : \omega_{i,\{1/e - r_i(G)\}}^{[j]} / \omega_{i,\{1/e - r_i(G)\}}^{[j-1]} \to \omega_{i,\{1/e - r_i(G)\}}^{[j-1]} / \omega_{i,\{1/e - r_i(G)\}}^{[j-2]}$ has a determinant of valuation $m_i^{[j]}(G)$ for all $2 \leq j \leq e$. The determinant of 
$$[\pi] : F_{i,\{1/e - r_i(G)\}}^{[j]} / F_{i,\{1/e - r_i(G)\}}^{[j-1]} \to F_{i,\{1/e - r_i(G)\}}^{[j-1]} / F_{i,\{1/e - r_i(G)\}}^{[j-2]}$$
has thus also a determinant of valuation $m_i^{[j]}(G)$ for all $2 \leq j \leq e$. Since $m_i^{[j]}(G) < 1/e - r_i(G)$, the invariant $m_i^{[j]}$ associated to the adequate filtration is equal to $m_i^{[j]}(G)$.
\end{proof}

These invariants $(m_i^{[j]} (G))$ depend only on the structure of $\omega$ as an $O_K \otimes_{\mathbb{Z}_p} O_F$-module. We have the following characterization of the Rapoport condition.

\begin{prop}
The $p$-divisible group $G$ satisfies the Rapoport condition if and only if $m_i^{[j]}(G)=0$ for all $1 \leq i \leq f$ and $2 \leq j \leq e$.
\end{prop}

\begin{proof}
The Rapoport condition is equivalent to the fact that each $\omega_{i}$ is free over $O_K \otimes_{O_{F^{ur}},\tau_i} O_F$ for all $1 \leq i \leq f$. Suppose that $G$ satisfies the Rapoport condition. Then we have 
$$\omega_{i} \simeq (O_K \otimes_{O_{F^{ur}},\tau_i} O_F)^d$$
as $O_K \otimes_{O_{F^{ur}},\tau_i} O_F$-module. One easily reduces to the case $d=1$. Since we want to prove that the invariants $m_i^{[j]}(G)$ are units, one can make the computation in the special fiber. But we have
$$\omega_{i} \otimes_{O_K} k \simeq k[X]/X^e$$
as $k \otimes_{\mathbb{Z}_p} O_F$-module, with $\pi$ acting on $k[X]/X^e$ by $X$. We get that $\omega_{i}^{[j]} \otimes_{O_K} k$ is generated as a $k$-vector space by $X^{e-j}, \dots, X^{e-1}$. The result follows. \\
Suppose now that we have $m_{i}^{[j]}(G)=0$ for all $1 \leq i \leq f$ and $2 \leq j \leq e$. The map $[\pi]$ acting on $\omega_{i}$ is then of the form
\begin{displaymath}
\left(
\begin{array}{cccc}
\sigma_{i,1} (\pi) I & M_i^{[2]}  & \dots & * \\
 & \sigma_{i,2} (\pi) I & \ddots  & * \\
& &  \ddots & M_i^{[e]} \\
& & & \sigma_{i,e} (\pi) I 
\end{array}
\right) \text{.} 
\end{displaymath}
All the blocks are of size $d$, $I$ is the identity matrix, and by assumption the matrixes $M_i^{[j]}$ are invertible. Let $v_1 \dots, v_d$ be a basis of the last block. Consider the family $(v_k, [\pi] v_k, \dots, [\pi]^{e-1} v_k)$, and let $N$ be the change-coordinate matrix for this family. The image of this matrix in the residue field $k$ is of the form
\begin{displaymath}
\left(
\begin{array}{ccccc}
0 & * & * & \dots & M_i^{[2]} \dots M_i^{[e]} \\
0 & * & *  & \reflectbox{$\ddots$} &\\
0 & * & M_i^{[e-1]} M_i^{[e]}  &  &\\
0 & M_i^{[e]} &   & &  \\
I &  &  & &  
\end{array}
\right) \text{.} 
\end{displaymath}
This matrix is invertible, so the previous family is a basis for $\omega_{i}$. This concludes the fact that this module is free over $O_K \otimes_{O_{F^{ur}},\tau_i} O_F$.
\end{proof}

\begin{rema}
If $G$ satisfies the Rapoport condition, there is only one adequate filtration on $\omega_{G,i,\{1/e\}}$ for all $1 \leq i \leq f$.
\end{rema}

We will now define another invariant related to the Verschiebung. 

\begin{prop} \label{hasse}
There exists a map $Hasse_i : \omega_{i,\{1/e\}}^{[1]} \to (\omega_{i-1,\{1/e\}} / \omega_{i-1,\{1/e\}}^{[e-1]})^{(p)}$. 
\end{prop}

\begin{proof}
Let $\mathcal{E}_{\{1\}}$ be the contravariant Dieudonné crystal of $\overline{G}$ evaluated at $O_{K,\{1\}}$ (\cite{BBM} section $3.3$). It is a free $O_{K,\{1\}}$-module of rank $efh$. It is endowed with an action of $O_F$ and we claim it is free of rank $h$ over $O_{K,\{1\}} \otimes_{\mathbb{Z}_p} O_F$. Let us justify briefly this assertion. It suffices to prove that $\mathcal{E}_{\{1\}} \otimes_{O_K} k$ is a free $k \otimes_{\mathbb{F}_p} O_F/p$-module. But this module lifts to a $W(k)$-module $\tilde{\mathcal{E}}$, where $W(k)$ is the ring of Witt vectors of $k$ ($\tilde{\mathcal{E}}$ is the classical Dieudonné module of $G \times_{O_K} k$). Since the module $\tilde{\mathcal{E}}$ is automatically free over $W(k) \otimes_{\mathbb{Z}_p} O_F$, this proves the claim. We have a decomposition
$$\mathcal{E}_{\{1\}} = \bigoplus_{i=1}^f \mathcal{E}_{i,\{1\}} \text{,} $$
with $O_{F^{ur}}$ acting on $\mathcal{E}_{i,\{1\}}$ by $\tau_i$. Let us denote by $[\pi]$ the action of $1 \otimes \pi$ on $\mathcal{E}_{\{1\}}$. Each $\mathcal{E}_{i,\{1\}}$ is a free $O_{K,\{1\}} [X]/ X^e$-module of rank $h$, with $X$ acting on it by $[\pi]$. Moreover, the Hodge filtration (\cite{BBM} corollary $3.3.5$) gives an exact sequence
$$0 \to \omega_{i,\{1\}} \to \mathcal{E}_{i,\{1\}} \to \omega_{G^D,i,\{1\}}^\vee \to 0 \text{,} $$
where $G^D$ is the Cartier dual of $G$. We have a Verschiebung map
$$V : \mathcal{E}_{\{1\}} \to \omega_{\{1\}}^{(p)} \text{.} $$
It induces maps 
$$V : \mathcal{E}_{i,\{1\}} \to (\omega_{i-1,\{1\}})^{(p)}$$ 
for all $1 \leq i \leq f$. We can now define the map of the proposition. Let $y \in \omega_{i,\{1/e\}}^{[1]}$; we then have $[\pi] y = 0$. We see $y$ as an element of $\mathcal{E}_{i,\{1/e\}}$, which is a free $O_{K,\{1/e\}} [X]/X^e$-module, with $X$ acting by $[\pi]$. Thus there exists $z \in \mathcal{E}_i$ such that $y = X^{e-1} z$; this element is defined modulo an element of $X \mathcal{E}_i$. Applying $V$, we get an element $Vz \in (\omega_{i-1,\{1/e\}})^{(p)}$. Since $X$ sends $(\omega_{i-1,\{1/e\}})^{(p)}$ into $(\omega_{i-1,\{1/e\}}^{[e-1]})^{(p)}$, the element 
$$Hasse_i(y) := Vz \in (\omega_{i-1,\{1/e\}} / \omega_{i-1,\{1/e\}}^{[e-1]})^{(p)}$$
is well defined.
\end{proof}

The map $Hasse_i$ can then be thought as the composition of the division by $[\pi]^{e-1}$ and the Verschiebung map. Taking the determinant of this map, on get an element $Hasse_i(G) \in O_{K,\{1/e\}}$. The valuation of this element will be noted $hasse_i(G) \in [0,1/e]$. Actually, each choice of an adequate filtration $(F_{i,\{1/e\}}^{[j]})$ for $\omega_{i,\{1/e\}}$ give a map
$$H_i : F_{i,\{1/e\}}^{[1]} \to (\omega_{i-1,\{1/e\}} / F_{i-1,\{1/e\}}^{[e-1]})^{(p)} \text{,} $$
and thus an element $hasse_i \in [0,1/e]$. Fortunately, this element does not depend on the adequate filtration under certain hypotheses.

\begin{prop} \label{indep_h}
Suppose that $hasse_i(G) + \max(r_i(G),r_{i-1}(G)) < 1/e$. Then the element $hasse_i(G)$ can be computed using any adequate filtration on $\omega_{i,\{1/e\}}$ and $\omega_{i-1,\{1/e\}}$.
\end{prop}

\begin{proof}
Let $r = \max(r_i(G),r_{i-1}(G))$, and $(F_{k,\{1/e\}}^{[j]})$ be adequate filtrations of $\omega_{k,\{1/e\}}$, for $k \in~\{i-~1,i\}$. From Proposition $\ref{independance}$, we get
$$F_{k,\{1/e-r\}}^{[j]} = \omega_{k,\{1/e-r\}}^{[j]}$$
for all $k \in \{i-1,i\}$ and $2 \leq j \leq e$. The map 
$$H_i : F_{i,\{1/e-r\}}^{[1]} \to (\omega_{i-1,\{1/e-r\}} / F_{i-1,\{1/e-r\}}^{[e-1]})^{(p)}$$
has thus a determinant of valuation $hasse_i(G)$. Since this element is strictly less than $1/e - r$, we can conclude. 
\end{proof}

In the ramified case, one can then construct the invariants $m_i^{[j]}(G)$ for $1 \leq i \leq f$, $2 \leq j \leq e$, which depend on the action of $O_F$ on $\omega_G$, and another invariant $hasse_i(G)$ related to the Verschiebung for $1 \leq i \leq f$. One can relate the unramified partial Hasse invariants to these ones.

\begin{prop} \label{compute_hasse}
The Verschiebung induce maps
$$V_i^{[j]} : \omega_{i,\{1/e\}}^{[j]} / \omega_{i,\{1/e\}}^{[j-1]} \to (\omega_{i-1,\{1/e\}}^{[j]} / \omega_{i-1,\{1/e\}}^{[j-1]})^{(p)}$$
for all $1 \leq i \leq f$ and $1 \leq j \leq e$. This map is equal to the composition
$$(M_{i-1}^{[j+1]})^{(p)} \circ \dots \circ (M_{i-1}^{[e]})^{(p)} \circ Hasse_i \circ M_i^{[2]} \circ \dots \circ M_i^{[j]} \text{.} $$
Let $Ha_i^{[j]}(G) \in O_{K,\{1/e\}}$ be the determinant of this map, and $ha_i^{[j]}(G) \in [0,1/e]$ its valuation.
We have the following equalities in $[0,1/e]$
$$ha_i^{[j]}(G) = hasse_i(G) + \sum_{k=2}^j m_i^{[k]}(G) + p \sum_{k=j+1}^e m_{i-1}^{[k]}(G)$$
$$ha_i(G) = e \cdot hasse_i(G) + \sum_{k=2}^e (e+1-k) m_i^{[k]}(G) + p \sum_{k=2}^e (k-1) m_{i-1}^{[k]}(G)$$
for all $1 \leq i \leq f$ and $1 \leq j \leq e$ (we say that the equality $a=b$ holds in $[0,r]$ if $\min(a,r) = \min(b,r)$).
\end{prop}

\begin{proof}
We first prove that the Verschiebung sends $\omega_{i,\{1/e\}}^{[j]}$ into $(\omega_{i-1,\{1/e\}}^{[j]})^{(p)}$. We keep the notations from the previous proposition. Let $y \in \omega_{i,\{1/e\}}^{[j]}$. We see $y$ as an element of $\mathcal{E}_{i,\{1/e\}}$; we have $X^j y =0$ so there exists $z \in \mathcal{E}_i$ such that $y = X^{e-j} z$. Therefore, we have $Vy = X^{e-j} Vz$. But $Vz \in (\omega_{i-1,\{1/e\}})^{(p)}$, and $X$ maps $(\omega_{i-1,\{1/e\}}^{[k]})^{(p)}$ into $(\omega_{i-1,\{1/e\}}^{[k-1]})^{(p)}$ for all $1 \leq k \leq e$. Thus 
$$Vy \in (\omega_{i-1,\{1/e\}}^{[j]})^{(p)} \text{.} $$
Since the Verschiebung respects the filtration on $\omega_{i,\{1/e\}}$, it induces maps on the graded pieces as claimed. \\
Let us write ${V_i^{[j]}}' = (M_{i-1}^{[j+1]})^{(p)} \circ \dots \circ (M_{i-1}^{[e]})^{(p)} \circ Hasse_i \circ M_i^{[2]} \circ \dots \circ M_i^{[j]}$. We will prove that ${V_i^{[j]}}' = V_i^{[j]}$. Let $y \in \omega_{i,\{1/e\}}^{[j]}$; then $y_1:=(M_i^{[2]} \circ \dots \circ M_i^{[j]})(y)$ is equal to $[\pi]^{j-1} y$. Since $[\pi]^j y = 0$, there exists $z \in \mathcal{E}_{i,\{1/e\}}$ such that $ y = X^{e-j} z$. Thus $y_1 = X^{e-1} z$, and $Hasse_i(y_1) = Vz$. Finally, we get
$${V_i^{[j]}}' (y) = [\pi]^{e-j} Hasse_i(y_1) = [\pi]^{e-j} Vz = V (X^{e-j} z) = Vy \text{.} $$
The rest of the equalities are obtained by taking the valuation of the determinant of the previous relation.
\end{proof}

We will also set for $1 \leq j \leq e$
$$ha^{[j]}(G) = \sum_{i=1}^f ha_i^{[j]}(G) \text{.} $$
From the previous proposition, we have 
$$ha^{[j]}(G) = \sum_{i=1}^f \left( hasse_i(G) + \sum_{k=2}^j m_i^{[k]}(G) + p \sum_{k=j+1}^e m_i^{[k]}(G)    \right ) \text{.} $$

The elements $(ha_i^{[j]}(G))_{i,j}$ will be called the \it{partial Hasse invariants}\normalfont. We will call the elements $(ha^{[j]}(G))_j$ the \it ramified partial Hasse invariants\normalfont. Finally, the elements $(hasse_i(G),m_i^{[j]}(G))_{i,j}$ will be called the \it{primitive Hasse invariants}\normalfont.

\begin{rema}
We have the following inequalities
$$ha^{[e]} \leq ha^{[e-1]} \leq \dots \leq ha^{[1]} \leq p \cdot ha^{[e]} \text{.} $$
\end{rema}

If the Hasse invariant is small enough, then so are the invariants $m_i^{[j]}(G)$ and $hasse_i(G)$. In particular, they do not depend on the choice of an ordering for the sets $\Sigma_i$ and can be computed using any adequate filtrations.

\begin{prop} \label{final_indep}
Suppose that $ha(G) < 1/e$. Then the elements $m_i^{[j]}(G)$ and $hasse_i(G)$ can be computed using any adequate filtrations on the modules $\omega_{i,\{1/e\}}$. 
\end{prop}

\begin{proof}
From the assumption $ha(G) < 1/e$, we easily get
$$ hasse_{i'} (G) + p \cdot r_i(G) < 1/e$$
for any elements $i,i'$ between $1$ and $f$. We get $1/e > p \cdot r_i(G) \geq 2 r_i(G) \geq r_i(G) + m_i^{[j]}(G)$ for any $1 \leq i \leq f$ and $2 \leq j \leq e$, so that the hypothesis of Corollary $\ref{indep_m}$ is satisfied. We also get $1/e >~hasse_i(G) +~\max(r_i(G),r_{i-1}(G))$ and the hypothesis of Proposition $\ref{indep_h}$ is satisfied for all $1 \leq i \leq f$.
\end{proof}

\subsection{Compatibility with duality} \label{dual}

The Hasse invariants we defined satisfy a compatibility with duality. We write $G^D$ for the Cartier dual of $G$. It is a $p$-divisible group over $O_K$ with an action of $O_F$. It has height $efh$ and dimension $ef(h-d)$. We start with the following lemma.

\begin{lemm} \label{dieudonne}
There exists a free $O_K \otimes_{\mathbb{Z}_p} O_F$-module $\mathcal{E}$ of rank $h$ with an exact sequence of $O_K \otimes_{\mathbb{Z}_p} O_F$-modules
$$0 \to \omega_G \to \mathcal{E} \to \omega_{G^D}^\vee \to 0   \text{.} $$
\end{lemm}

\begin{proof}
For all integer $n \geq 1$, let $\mathcal{E}_{\{n\}}$ be the contravariant Dieudonné crystal of $G \times_{O_K} O_{K,\{n\}}$ evaluated at $O_{K,\{n\}}$ (\cite{BBM} section $3.3$). It is a free $O_{K,\{n\}}$-module of rank $efh$ with an action of $O_F$. As we have seen in the proof of Proposition $\ref{hasse}$, it is free as an $O_{K,\{n\}} \otimes_{\mathbb{Z}_p} O_F$-module. Define
$$\mathcal{E} := \varprojlim_n \mathcal{E}_{\{n\}} \text{.} $$
It is a free $O_K \otimes_{\mathbb{Z}_p} O_F$-module of rank $h$. The Hodge filtration (\cite{BBM} corollary $3.3.5$) gives exact sequences for all integer $n \geq 1$
$$0 \to \omega_{G,\{n\}} \to \mathcal{E}_{\{n\}} \to \omega_{G^D,\{n\}}^\vee \to 0  \text{.}  $$
This concludes the proof.
\end{proof}

We now state the duality property verified by the Hasse invariants.

\begin{prop}
We have the equalities $ha(G)=ha(G^D)$ and $ha_i(G) = ha_i(G^D)$ for all $1 \leq~i \leq~f$. We also have $m_i^{[j]}(G) = m_i^{[j]}(G^D)$ for all $1 \leq i \leq f$ and $2 \leq j \leq e$. Suppose moreover that $ha(G) < 1/e$. Then $hasse_i(G) = hasse_i(G^D)$ for all $1 \leq i \leq f$.
\end{prop}

\begin{proof}
The relation $ha(G) = ha(G^D)$ is proved in \cite{Fa} Proposition $2$. The same proof (decomposing each module according to the elements of $\mathcal{T}$) gives the equalities $ha_i(G) =~ha_i(G^D)$ for $1 \leq i \leq f$. \\
We will now prove that $m_i^{[j]} (G^D) = m_i^{[j]} (G)$ for all $1 \leq i \leq f$ and $2 \leq j \leq e$. This will allow us to conclude thanks to the relation in $[0,1/e]$
$$ha_i(G) = e \cdot hasse_i(G) + \sum_{k=2}^e (e+1-k) m_i^{[k]}(G) + p \sum_{k=2}^e (k-1) m_{i-1}^{[k]}(G) \text{.} $$
Let us fix an integer $i$ between $1$ and $f$. The free $O_K \otimes_{\mathbb{Z}_p} O_F$-module $\mathcal{E}$ decomposes in 
$$\mathcal{E} = \bigoplus_{i=1}^f \mathcal{E}_i \text{,} $$
where $\mathcal{E}_i$ is a free $O_K \otimes_{O_{F^{ur}},\tau_i} O_F$-module. Note the equality
$$O_K \otimes_{O_{F^{ur}},\tau_i} O_F = O_K[X] / \prod_{k=1}^e (X - \sigma_{i,k}(\pi)) \text{.} $$
We will denote by $[\pi]$ the action of $\pi$ in $\mathcal{E}_i$. Recall the exact sequence
$$0 \to \omega_{G,i} \to \mathcal{E}_i \to \omega_{G^D,i}^\vee \to 0   \text{.} $$
To ease the notations, let us write $\pi_k:= \sigma_{i,k}(\pi)$ for $1 \leq k \leq e$. Let us define
$$\mathcal{F}_i^{[j]} := \left\{ y \in \mathcal{E}_i, \left( \prod_{k=j+1}^e ([\pi] - \pi_k) \right) \cdot y \in \omega_{G,i}^{[j]} \right\}$$
for $0 \leq j \leq e$. The module $\mathcal{F}_i^{[j]}$ is free of rank $he-j(h-d)$ over $O_K$. Furthermore, since $([\pi] - \pi_k) \omega_{G,i}^{[k]} \subset \omega_{G,i}^{[k-1]}$ for all $1 \leq k \leq e$, we have inclusions
$$0 \subset \omega_{G,i}^{[1]} \subset \dots \subset \omega_{G,i}^{[e]} = \mathcal{F}_i^{[e]} \subset \dots \subset \mathcal{F}_i^{[0]} = \mathcal{E}_i \text{.} $$
Moreover, the map $[\pi]$ acts by $\pi_j$ on $\mathcal{F}_i^{[j-1]} / \mathcal{F}_i^{[j]}$ for all $1 \leq j \leq e$. We thus have a filtration
$$0 \subset (\mathcal{E}_i / \mathcal{F}_i^{[1]})^\vee \subset \dots \subset (\mathcal{E}_i / \mathcal{F}_i^{[e]})^\vee = \omega_{G^D,i} \text{,} $$
with $(\mathcal{E}_i / \mathcal{F}_i^{[j]})^\vee$ free of rank $j(h-d)$, and $[\pi]$ acts on the quotient $(\mathcal{E}_i / \mathcal{F}_i^{[j]})^\vee / (\mathcal{E}_i / \mathcal{F}_i^{[j-1]})^\vee$ by $\pi_j$, for $1 \leq j \leq e$. This proves that $(\mathcal{E}_i / \mathcal{F}_i^{[j]})^\vee = \omega_{G^D,i}^{[j]}$, or in other terms, 
$$\mathcal{F}_i^{[j]} / \omega_{G,i} = (\omega_{G^D,i} / \omega_{G^D,i}^{[j]})^\vee \text{.} $$
We have thus related the filtration on $\omega_{G^D,i}$ to the one on $\omega_{G,i}$. We want to compute the element $m_i^{[2]}(G^D)$. For this, one can work with $\mathcal{E}_{i,\{1/e\}}$, which is a free $O_{K,\{1/e\}}[X]/X^e$-module, with $X$ acting by $[\pi]$. Note that since $\omega_{G,i}^{[j]}$ is contained in the set of elements killed by $\prod_{k=1}^j ([\pi] - \pi_k)$, we have
$$\mathcal{F}_{i,\{1/e\}}^{[j]} = \left\{ y \in \mathcal{E}_{i,\{1/e\}} ,  X^{e-j} y \in \omega_{G,i,\{1/e\}}^{[j]} \right\} \text{.} $$
The action of $[\pi]$ on $\omega_{G,i,\{1/e\}}^{[2]}$ is of the form
 \begin{displaymath}
\left(
\begin{array}{cc}
0 & M_i^{[2]} \\
& 0 
\end{array}
\right) \text{,} 
\end{displaymath}
with the valuation of the determinant of $M_i^{[2]}$ equal to $m_i^{[2]}(G)$. From the elementary divisors theorem for valuation rings, one can moreover suppose that $M_i^{[2]}$ is diagonal. Let us write $y_1, \dots, y_d$ the diagonal coefficients; we order them so that $y_1, \dots, y_{r}$ are not units, and $y_{r+1}, \dots, y_d$ are. We can thus find a basis $(e_1, \dots, e_{2d})$ of $\omega_{G,i,\{1/e\}}^{[2]}$ such that $\omega_{G,i,\{1/e\}}^{[1]}$ is generated by $(e_1, \dots, e_d)$, and $[\pi] e_{d+k}=y_k e_k$ for all $1 \leq k \leq d$. One may then find a basis $(\varepsilon_1, \dots, \varepsilon_h)$ of $\mathcal{E}_{i,\{1/e\}}$ over $O_{K,\{1/e\}}[X] /X^e$ such that 
\begin{itemize}
\item $e_k = X^{e-1} \varepsilon_k$ for $1 \leq k \leq d$.
\item $e_{d+k} = X^{e-1} \varepsilon_{d+k} + X^{e-2} y_{k} \varepsilon_{k}$ for $1 \leq k \leq r$.
\item $e_{d+k} = X^{e-2} y_k \varepsilon_k$ for $r+1 \leq k \leq d$.
\end{itemize}
Note that one has necessarily $d+r \leq h$. We then see that $\mathcal{F}_{i,\{1/e\}}^{[1]}$ is generated by $X \mathcal{E}_{i,\{1/e\}}$ and $(\varepsilon_k)_{1 \leq k \leq d}$. The module $\mathcal{F}_{i,\{1/e\}}^{[2]}$ is generated by $X^2 \mathcal{E}_{i,\{1/e\}}$ and $(X \varepsilon_k)_{1 \leq k \leq d}$, $(X \varepsilon_{d+k} + y_k \varepsilon_k)_{1 \leq k \leq r}$, $(y_k \varepsilon_k)_{r+1 \leq k \leq d}$. We then may take $(\varepsilon_{d+1}, \dots, \varepsilon_{h})$ for a basis of $\mathcal{E}_{i,\{1/e\}} / \mathcal{F}_{i,\{1/e\}}^{[1]}$, and $(\varepsilon_1, \dots, \varepsilon_r, X \varepsilon_{d+r+1}, \dots, X \varepsilon_h)$ for a basis of $\mathcal{F}_{i,\{1/e\}}^{[1]} / \mathcal{F}_{i,\{1/e\}}^{[2]}$. With these bases, the matrix of $[\pi] : \mathcal{E}_{i,\{1/e\}} / \mathcal{F}_{i,\{1/e\}}^{[1]} \to \mathcal{F}_{i,\{1/e\}}^{[1]} / \mathcal{F}_{i,\{1/e\}}^{[2]}$ is equal to
\begin{displaymath}
\left(
\begin{array}{cccccc}
- y_1 & & & & & \\
& \ddots & & & & \\
& & - y_r & & & \\
& & & 1 & & \\
& & & & \ddots & \\
& & & & & 1
\end{array}
\right) \text{.} 
\end{displaymath}
Indeed, we have the relation $X \varepsilon_{d+k} + y_k \varepsilon_k =0$ in $\mathcal{F}_{i,\{1/e\}}^{[1]} / \mathcal{F}_{i,\{1/e\}}^{[2]}$ for all $1 \leq k \leq r$. In particular, the determinant of this matrix has valuation $m_i^{[2]}(G)$. This proves that $m_i^{[2]}(G) = m_i^{[2]}(G^D)$. \\
Considering $\mathcal{F}_{i,\{1/e\}}^{[1]} / \omega_{G,i,\{1/e\}}^{[1]}$, which is a free $O_{K,\{1/e\}}[X]/X^{e-1}$-module of rank $h$, one can prove by induction that the action of $[\pi]$ on $\omega_{G^D,i,\{1/e\}}^\vee$ is of the form
 \begin{displaymath}
\left(
\begin{array}{cccc}
0 & {M_i^{[e]}}' & \dots & * \\
 & 0 & \ddots & \vdots \\
 & & \ddots & {M_i^{[2]}}' \\
& & & 0
\end{array}
\right) \text{,} 
\end{displaymath}
with the property that the determinant of ${M_i^{[j]}}'$ has valuation $m_i^{[j]}(G)$ for $2 \leq j \leq e$. This concludes the proof.
\end{proof}

\subsection{Partial Hasse invariants in family}

Let $S$ be an $O_K$-scheme. In this section only, we will consider a $p$-divisible group $G \to S$ of height $efh$ with an action of $O_F$. Let $\omega_{G/S}$ be the conormal sheaf of $G$ along its unit section; it is a locally free sheaf over $S$. It also has an action of $O_F$, and thus decomposes into $\omega_{G/S} = \oplus_{i=1}^f \omega_{G/S,i}$. We also make the following hypothesis.

\begin{hypo}
For each integer $1 \leq i \leq f$, there exists a filtration 
$$0 = \omega_{G/S,i}^{[0]} \subset \omega_{G/S,i}^{[1]} \subset \dots \subset \omega_{G/S,i}^{[e]} = \omega_{G/S,i} \text{,} $$
such that for all $1 \leq j \leq e$, $\omega_{G/S,i}^{[j]} / \omega_{G/S,i}^{[j-1]}$ is a locally free sheaf of rank $d$, and $O_F$ acts by $\sigma_{i,j}$ on it.
\end{hypo}

This hypothesis is for example satisfied when one considers certain moduli spaces of abelian varieties satisfying the Pappas-Rapoport condition (\cite{P-R}). It implies that the dimension of $G$ over $S$ is equal to $def$. Each $\omega_{G/S,i}$ is then a locally free sheaf of rank $ed$. We will also define
$$\mathcal{L}_{S,i}^{[j]} := \det (\omega_{G/S,i}^{[j]} / \omega_{G/S,i}^{[j-1]})$$
for all $1 \leq i \leq f$ and $1 \leq j \leq e$. It is an invertible sheaf over $S$. We will define the partial Hasse invariants as sections of certain products of these invertible sheaves. For this, we need to work over $O_{K,\{1/e\}}$. Let $S_{\{1/e\}} := S \times_{O_K} O_{K,\{1/e\}}$ 

\begin{prop}
The $Verschiebung$ map induces sections 
$$Ha_i^{[j]} \in H^0(S_{\{1/e\}},(\mathcal{L}_{S_{\{1/e\}},i-1}^{[j]})^{\otimes p} \otimes  (\mathcal{L}_{S_{\{1/e\}},i}^{[j]})^{-1})$$
for all $1 \leq i \leq f$ and $1 \leq j \leq e$. The primitive Hasse invariants are sections
$$Hasse_i \in H^0(S_{\{1/e\}},(\mathcal{L}_{S_{\{1/e\}},i-1}^{[e]})^{\otimes p} \otimes  (\mathcal{L}_{S_{\{1/e\}},i}^{[1]})^{-1})$$
and 
$$M_i^{[j]} \in H^0(S_{\{1/e\}},\mathcal{L}_{S_{\{1/e\}},i}^{[j-1]} \otimes  (\mathcal{L}_{S_{\{1/e\}},i}^{[j]})^{-1})$$
for $1 \leq i \leq f$ and $2 \leq j \leq e$. Moreover, one has the relations for all $1 \leq i \leq f$ and $1 \leq j \leq e$
$$Ha_i^{[j]} = (M_{i-1}^{[j+1]})^{ p} \dots (M_{i-1}^{[e]})^{ p} \cdot  Hasse_i \cdot M_i^{[2]} \dots M_i^{[j]}$$
\end{prop}

\begin{proof}
This is exactly the construction done in \cite{R-X}.
\end{proof}

\section{Partial degrees}

\subsection{Definitions}

We are still considering a $p$-divisible $G$ endowed with an action of $O_F$ satisfying Hypothesis $\ref{hypothese}$. Let $N \geq 1$ be an integer, and let $H$ be a finite flat subgroup of $G[p^N]$ stable by $O_F$. Its height is thus a multiple of $f$, that we write $fh_0$. Let $\omega_H$ be the conormal sheaf of $H$ along its unit section; it is a finitely generated $O_K$-module of $p^N$-torsion. We have an exact sequence
$$0 \to \omega_{G/H} \to \omega_{G} \to \omega_H \to 0 \text{.} $$

The degree of $H$ (defined in \cite{Fa_deg}), written $\deg H$ can be defined as the valuation of the determinant of the map $\omega_{G/H} \to \omega_G$. Alternatively, we have $\deg H = v(Fitt_0$ $\omega_H)$, where $Fitt_0$ is the Fitting ideal, and the valuation of an ideal $x O_K$ is the valuation of $x$. \\
\indent The definition of the partial degrees according to the elements of $\mathcal{T}$ is very natural. It has already been done in \cite{Bi}. We have a decomposition $\omega_H = \oplus_{i=1}^f \omega_{H,i}$, where $O_{F^{ur}}$ acts on $\omega_{H,i}$ by $\tau_i$, and exact sequences
$$0 \to \omega_{G/H,i} \to \omega_{G,i} \to \omega_{H,i} \to 0$$
for all $1 \leq i \leq f$.

\begin{defi}
The \textit{unramified partial degree} $\deg_i H$ is defined as the valuation of the determinant of the map $\omega_{G/H,i} \to \omega_{G,i}$ for all $1 \leq i \leq f$. Alternatively, we have $\deg_i H = v(Fitt_0$ $\omega_{H,i})$.
\end{defi}

\begin{ex}
We have $\deg_i G[p^N] = Ned$ and $\deg_i G[\pi^N] = Nd$ for all $1 \leq i \leq f$. If $H$ is multiplicative, then $\deg_i H =h_0$ for all $i$; if $H$ is étale, then $\deg_i H = 0$ for all $i$.  
\end{ex}

The unramified partial degree of $H^D$ can be defined either by using the module $\omega_{H^D,i}$ or the map $\omega_{G^D,i} \to \omega_{(G/H)^D,i}$ for all $1 \leq i \leq f$. The unramified partial degrees are thus canonically defined, and depend only on the subgroup $H$ and not on the $p$-divisible group $G$.. This will not be the case for the general partial degrees. \\ 
\indent We will now refine the unramified partial degrees, to take into account the full action of $O_F$. These general partial degrees have already been defined by Sasaki for the Hilbert modular variety (\cite{Sa}), although very few properties were known. We recall that we have filtrations $(\omega_{G,i}^{[j]})_{1 \leq j \leq e}$ and $(\omega_{G/H,i}^{[j]})_{1 \leq j \leq e}$. The map $\omega_{G/H,i} \to \omega_{G,i}$ respect this filtration; we thus get maps
$$\omega_{G/H,i}^{[j]}  \to \omega_{G,i}^{[j]}$$
for all $1 \leq i \leq f$ and $1 \leq j \leq e$.

\begin{defi}
The \textit{partial degree} $\deg_i^{[j]} H$ is defined as the valuation of the determinant of the map 
$$\omega_{G/H,i}^{[j]} / \omega_{G/H,i}^{[j-1]} \to \omega_{G,i}^{[j]} / \omega_{G,i}^{[j-1]}$$
for all $1 \leq i \leq f$ and $1 \leq j \leq e$.
\end{defi}

\begin{ex}
We have $\deg_i^{[j]} G[\pi^N] = N d/e$ for all $1 \leq i \leq f$ and $1 \leq j \leq e$.
\end{ex}

Define $\omega_{H,i}^{[j]}$ as the image of $\omega_{G,i}^{[j]}$ in $\omega_H$ for all $1 \leq i \leq f$ and $1 \leq j \leq e$. Then we also have $\deg_i^{[j]} H = v(Fitt_0$ $(\omega_{H,i}^{[j]} / \omega_{H,i}^{[j-1]}))$ for $1 \leq i \leq f$ and $1 \leq j \leq e$. We define the element $\deg_i^{[j]} H^D$ as the valuation of the determinant of the map
$$\omega_{G^D,i}^{[j]} / \omega_{G^D,i}^{[j-1]} \to \omega_{(G/H)^D,i}^{[j]} / \omega_{(G/H)^D,i}^{[j-1]} \text{.} $$

\begin{rema}
One can define the partial degrees of a finite flat $O_F$-stable subgroup $H \subset G[p^N]$ even if the $p$-divisible group $G$ does not satisfy Hypothesis $\ref{hypothese}$.
\end{rema}

\subsection{Properties} \label{deg_prop}

The unramified partial degrees enjoy the following properties.

\begin{prop}
Let $H$ be an $O_F$-stable finite flat subgroup of $G[p^N]$ of height $fh_0$.
\begin{itemize}
\item We have $\deg H = \sum_{i=1}^h \deg_i H$.
\item The unramified partial degrees are additive : if $H_1 \subset H_2$ are two finite flat $O_F$-stable subgroups of $G[p^N]$, then
$$\deg_i H_2 = \deg_i H_1 + \deg_i H_2 / H_1$$
for all $1 \leq i \leq f$.
\item We have $\deg_i H^D = h_0 - \deg_i H$ for all $1 \leq i \leq f$.
\item The unramified partial degree $\deg_i H$ is in $[0,h_0]$ for all $1 \leq i \leq f$.
\end{itemize}
\end{prop}

\begin{proof}
The first relation comes from the decomposition $\omega_H = \oplus_{i=1}^f \omega_{H,i}$. The second relation is implied by the exact sequences
$$0 \to \omega_{H_2/H_1,i} \to \omega_{H_2,i} \to \omega_{H_1,i} \to 0$$
for $1 \leq i \leq f$. For the third equation, one reduces to the case where $H$ is $p$-torsion by additivity. Let $\mathcal{E}_{G}$ and $\mathcal{E}_{G/H}$ be the free $O_K$-modules constructed in Lemma \ref{dieudonne} for $G$ and $G/H$ respectively. Let $\mathcal{E}_{H}$ be the Dieudonné crystal associated to $H$ evaluated at $O_K/p$ (\cite{BBM} section $3.1$). Finally, let $\nu_{H^D}$ be the cokernel of the map $\omega_{(G/H)^D}^\vee \to \omega_{G^D}^\vee$. We have a commutative diagram

\begin{displaymath}
\xymatrix{
 & 0 \ar[d] & 0 \ar[d] & 0 \ar[d] & \\
0 \ar[r] & \omega_{G/H} \ar[r] \ar[d] & \omega_G \ar[r] \ar[d] & \omega_H \ar[r] \ar[d] & 0 \\
0 \ar[r] & \mathcal{E}_{G/H} \ar[r] \ar[d] & \mathcal{E}_G \ar[r] \ar[d] & \mathcal{E}_H \ar[r] \ar[d] & 0 \\
0 \ar[r] & \omega_{(G/H)^D}^\vee \ar[r] \ar[d] & \omega_{G^D}^\vee \ar[r] \ar[d] & \nu_{H^D} \ar[r] \ar[d] & 0 \\
 & 0 & 0 & 0 &
} 
\end{displaymath}
Moreover, one checks that the horizontal and vertical lines are exact sequences. The modules $\mathcal{E}_H$ and $\nu_{H^D}$ have an action of $O_F$, and thus decompose in $\mathcal{E}_H = \oplus_{i=1}^f \mathcal{E}_{H,i}$, $\nu_{H^D} = \oplus_{i=1}^f \nu_{H^D,i}$. We have an exact sequence
$$0 \to \omega_{H,i} \to \mathcal{E}_{H,i} \to \nu_{H^D,i} \to 0$$
for all $i$ between $1$ and $f$. We deduce the third equality, since $\deg_i H^D =  v(Fitt_0$ $\nu_{H^D,i})$ and $\mathcal{E}_{H,i}$ is a free $O_{K}/p$-module of rank $h_0$. From this relation one can easily deduce the last assertion.
\end{proof}

The properties verified by the general partial degrees are similar, but the proofs of these properties are more difficult.

\begin{prop}
Let $H \subset G[p^N]$ be a finite flat $O_F$-stable subgroup, and let $i$ be an integer between $1$ and $f$. 
\begin{itemize}
\item We have $\sum_{j=1}^e \deg_i^{[j]} H = \deg_i H$. \\
\item The partial degrees are additive : if $H_1 \subset H_2$ are two finite flat $O_F$-stable subgroups of $G[p^N]$, then
$$\deg_i^{[j]} H_2 = \deg_i^{[j]} H_1 + \deg_i^{[j]} H_2 / H_1$$
for all $1 \leq j \leq e$.
\item We have $\deg_i^{[j]} H^D = h_0/e - \deg_i^{[j]} H$ for all $1 \leq j \leq e$.
\item The partial degree $\deg_i^{[j]} H$ is in $[0,h_0/e]$ for all $1 \leq j \leq e$.
\end{itemize}
\end{prop}

\begin{proof}
For the first assertion, one has just to observe that the determinant of the map $\omega_{G/H,i} \to \omega_{G,i}$ is the product of the determinant on the graded pieces $\omega_{G/H,i}^{[j]} / \omega_{G/H,i}^{[j-1]} \to \omega_{G,i}^{[j]} / \omega_{G,i}^{[j-1]}$. \\
The second relation is obtained by remarking that the map $\omega_{G/H_2,i} \to \omega_{G,i}$ factorizes as 
$$\omega_{G/H_2,i} \to \omega_{G/H_1,i} \to \omega_{G,i} \text{,} $$
and that this factorization respects the filtrations on the three modules. \\
Let us now prove the third relation. Note that it implies the last assertion. By additivity, one reduces to the case where $H$ is a subgroup of $G[\pi]$. Let $\mathcal{E}$ be the free $O_K \otimes_{\mathbb{Z}_p} O_F$-module constructed in the section $\ref{dual}$. We keep the notations from that section. Let us fix an integer $i$ between $1$ and $f$. The module $\mathcal{E}_i$ is a free $O_K [X] / \prod_{k=1}^e (X - \sigma_{i,k}(\pi))$-module, with $X$ acting by $1 \otimes \pi$. To simplify the notations, let us write $\pi_k := \sigma_{i,k}(\pi)$ for all $1 \leq k \leq e$. Recall that we have exact sequences coming from the Hodge filtration
$$0 \to \omega_{G,i} \to \mathcal{E}_i \to \omega_{G^D,i}^\vee \to 0 \text{.} $$
Moreover we have the following filtration on $\mathcal{E}_i$ : 
$$0 \subset \omega_{G,i}^{[1]} \subset \dots \subset \omega_{G,i}^{[e]} = \mathcal{F}_i^{[e]} \subset \dots \subset \mathcal{F}_i^{[0]} = \mathcal{E}_i \text{,} $$
with $\mathcal{F}_i^{[j]} / \omega_{G,i} = (\omega_{G^D,i} / \omega_{G^D,i}^{[j]})^\vee$ for all $1 \leq j \leq e$. Let $\mathcal{E}_{G/H,i}$ be the module constructed for $G/H$; it is a free $O_K [X] / \prod_{k=1}^e (X - \pi_k)$ contained in $\mathcal{E}_i$ and containing $X \mathcal{E}_i$. Moreover, $\mathcal{E}_i / X \mathcal{E}_i \simeq (O_K /p )^{h}$, and the quotient $\mathcal{E}_i / \mathcal{E}_{G/H,i}$ is free over $O_K/p$ of rank $h_0$. The module $\mathcal{E}_{G/H,i} / X \mathcal{E}_i$ is thus a direct factor of $\mathcal{E}_i / X \mathcal{E}_i$.\\
Let $\varepsilon_1, \dots, \varepsilon_h$ be a basis of $\mathcal{E}_i$ over $O_K [X] / \prod_{k=1}^e (X - \pi_k)$, such that $\omega_{G,i}^{[1]}$ is generated by 
$$\prod_{k=2}^e (X-~\pi_k) \varepsilon_1, \dots, \prod_{k=2}^e (X - \pi_k) \varepsilon_d \text{.} $$
Thus, $\mathcal{F}_i^{[1]}$ is generated by $(X-\pi_1) \mathcal{E}_i$ and $\varepsilon_1, \dots, \varepsilon_d$. Let $\overline{\omega_{G,i}^{[1]}}$ and $\overline{\mathcal{F}_i^{[1]}}$ be respectively the images of $\omega_{G,i}^{[1]}$ and $\mathcal{F}_i^{[1]}$ in $\mathcal{E}_i / X \mathcal{E}_i$.  We then have
\begin{itemize}
\item the module $\overline{\omega_{G,i}^{[1]}}$ is generated by $(\prod_{k=2}^e \pi_k) \varepsilon_1, \dots, (\prod_{k=2}^e \pi_k) \varepsilon_d$
\item the module $\overline{\mathcal{F}_i^{[1]}}$ is generated by $\varepsilon_1, \dots, \varepsilon_d, \pi_1 \varepsilon_{d+1}, \dots, \pi_1 \varepsilon_h$.
\end{itemize}
The quotient $\overline{\mathcal{F}_i^{[1]}} / \overline{\omega_{G,i}^{[1]}}$ is thus a free $O_K / \prod_{k=2}^e \pi_k$-module of rank $h$. \\
Let $\mathcal{E}_{H,i}$ be the quotient $\mathcal{E}_i / \mathcal{E}_{G/H,i}$; it is a free $O_K/p$-module of rank $h_0$. The image of $\omega_{G,i}^{[1]}$ in $\mathcal{E}_{H,i}$ is equal to $\omega_{H,i}^{[1]}$, and if $\mathcal{F}_{H,i}^{[1]}$ is the image of $\mathcal{F}_{i}^{[1]}$, then $\mathcal{E}_{H,i} / \mathcal{F}_{H,i}^{[1]}$ is isomorphic to $(\omega_{H^D,i}^{[1]})^\vee$. To sum up, one has a filtration
$$0 \subset \omega_{H,i}^{[1]} \subset \mathcal{F}_{H,i}^{[1]} \subset \mathcal{E}_{H,i} \text{.} $$
From the calculations made above, one sees that $\mathcal{F}_{H,i}^{[1]} / \omega_{H,i}^{[1]}$ is a free $O_K / \prod_{k=2}^e \pi_k$ module of rank $h_0$. This implies the relation
$$\deg_i^{[1]} H + h_0(1 - \frac{1}{e}) + \deg_i^{[1]} H^D = h_0 \text{.} $$
This gives the relation for the first partial degree. Considering the module $\mathcal{F}_i^{[1]} / \omega_{G,i}^{[1]}$, which is free over $O_K[X] / \prod_{k=2}^e (X - \pi_k)$, one gets all the other relations by induction.
\end{proof}

The elements $(\deg_i^{[j]} H)_{1 \leq j \leq e}$ depend on a choice of an ordering for the set $\Sigma_i$. However, we have the following property.

\begin{prop} \label{indep_deg}
Let $H$ be a finite flat $O_F$-stable subgroup of $G[p^N]$, and let $i$ be an integer between $1$ and $f$. Suppose that $\min ( \deg_i H , \deg_i H^D) + r_i(G) < 1/e$. Then the elements $(\deg_i^{[j]} H)_{1 \leq j \leq e}$ do not depend on any choice. 
\end{prop}

\begin{proof}
Suppose that $\deg_i H + r_i(G) < 1/e$. The elements $(\deg_i^{[j]} H)_{1 \leq j \leq e}$ can be computed using the modules $\omega_{H,i}^{[1]}, \dots, \omega_{H,i}^{[e]}$. We recall that for all $1 \leq j \leq e$, the module $\omega_{H,i}^{[j]}$ is the image of $\omega_{G,i}^{[j]}$ in $\omega_H$. Since $\deg_i H < 1/e - r_i(G)$, this module is also the image of $\omega_{G,i,\{1/e - r_i(G)\}}^{[j]}$ in $\omega_H$. But Proposition $\ref{independance}$ tells that the modules $(\omega_{G,i,\{1/e - r_i(G)\}}^{[j]})_{1 \leq j \leq e}$ do not depend on any choice. We conclude that so do the modules $\omega_{H,i}^{[j]}$, and the elements $\deg_i^{[j]} H$, for $1 \leq j \leq e$. \\
Suppose now that $\deg_i H^D + r_i(G) < 1/e$. Since $r_i(G) = r_i(G^D)$, the previous argument shows that the elements $\deg_i^{[j]} H^D$ are well defined for all $j$ between $1$ and $e$. The formula $\deg_i^{[j]} H =~h_0/e -~\deg_i^{[j]} H^D$ implies that the elements $\deg_i^{[j]} H$ are well defined too for all $1 \leq j \leq e$.
\end{proof}

\subsection{Partial degrees in family}

Let $\mathfrak{X}$ be an admissible formal $O_K$-scheme (\cite{Bo} section $2.4$). In this section only, $G$ will denote a $p$-divisible group over $\mathfrak{X}$ of height $efh$, with an action of $O_F$. Suppose also that there is a finite flat subgroup $H \subset G[p^N]$ over $\mathfrak{X}$ for some integer $N$. We will denote the $p$-divisible group $G/H$ by $G'$, and $\phi : G \to G'$ the isogeny. Let $\omega_{G/\mathfrak{X}}$ and $\omega_{G'/\mathfrak{X}}$ denote the conormal sheaves of $G$ and $G'$ along their unit sections. They are locally free sheaves over $\mathfrak{X}$, and thus decompose according to the elements of $\mathcal{T}$ 
$$ \omega_{G/\mathfrak{X}} = \bigoplus_{i=1}^f \omega_{G/\mathfrak{X},i} \text{,} $$
and similarly for $\omega_{G'/\mathfrak{X}}$. We thus have a map $\phi^* : \omega_{G'/\mathfrak{X},i} \to \omega_{G/\mathfrak{X},i}$.
We will also make the following hypothesis.

\begin{hypo}
For each $p$-divisible $G_0$ equal to $G$ or $G'$, and for each integer $1 \leq i \leq f$, there exists a filtration 
$$0 = \omega_{G_0/\mathfrak{X},i}^{[0]} \subset \omega_{G_0/\mathfrak{X},i}^{[1]} \subset \dots \subset \omega_{G_0/\mathfrak{X},i}^{[e]} = \omega_{G_0/\mathfrak{X},i} \text{,} $$
such that $\omega_{G_0/\mathfrak{X},i}^{[j]} / \omega_{G/\mathfrak{X},i}^{[j-1]}$ is a locally free sheaf of rank $d$, and $O_F$ acts by $\sigma_{i,j}$ on it for all $1 \leq j \leq e$. Moreover, the map $\phi^*$ respects these filtrations.
\end{hypo}

We will define
$$\mathcal{L}_{G,i}^{[j]} := \det (\omega_{G/\mathfrak{X},i}^{[j]} / \omega_{G/\mathfrak{X},i}^{[j-1]}) \text{,} $$
and similarly for $\mathcal{L}_{G',i}^{[j]}$ for all $1 \leq i \leq f$ and $1 \leq j \leq e$. They are invertible sheaves over $\mathfrak{X}$. The map $\phi^*$ give sections
$$\delta_{H,i}^{[j]} \in H^0 \left (\mathfrak{X},\mathcal{L}_{G,i}^{[j]} \cdot {\mathcal{L}_{G',i}^{[j]}}^{-1} \right)$$
for all $1 \leq i \leq f$ and $1 \leq j \leq e$. Let $X^{rig}$ be the generic fiber of $\mathfrak{X}$ in the sense of Raynaud (\cite{Bo} section $2.7$). We will still denote by $\mathcal{L}_{G,i}^{[j]}$ the invertible sheaves on $X^{rig}$, and by $\delta_{H,i}^{[j]}$ the sections induced on $X^{rig}$. Moreover, we have a norm map (\cite{Ka} section $2$)
\begin{equation*}
\begin{aligned}
X^{rig} & \to \mathbb{R} \\
x & \to  | \delta_{H,i}^{[j]}(x) |
\end{aligned} \text{.} 
\end{equation*}

\begin{defi}
Let $1 \leq i \leq f$ and $1 \leq j \leq e$. The partial degree $\deg_i^{[j]} H$ is a function $X^{rig} \to \mathbb{R}$ defined by
$$ | \delta_{H,i}^{[j]} (x) | = p^{- \deg_i^{[j]} H(x)} \text{,} $$
for all $x \in X^{rig}$.
\end{defi}

\section{The canonical subgroup}

\subsection{An alternative approach}

We recall the main theorem of Fargues (\cite{Fa}) regarding the construction of the canonical subgroup.

\begin{theo}[Fargues] \label{main_theo}
Suppose $p \neq 2$, and let $G_0$ be a $p$-divisible group of height $h_0$ and dimension $d_0$ over $O_K$. We suppose that $ha(G_0) < 1/2$, and that $ha(G_0) < 1/3$ if $p=3$. Then there exists a canonical subgroup $C_{0} \subset G_0[p]$, such that : 
\begin{itemize}
\item $C_{0}$ has height $d_0$.
\item $\deg C_0^D = ha(G_0)$.
\item $C_0$ is the kernel of the Frobenius in $G \times_{O_K} O_{K,\{1 - ha(G_0)\}}$.
\item if $ha(G_0) < 1/(p+1)$ then we have $ha(G_0/C_0) = p \cdot ha(G_0)$.
\end{itemize}
Moreover the construction of the canonical subgroup is compatible with duality.
\end{theo}

This theorem says that if a $p$-divisible group is close to being ordinary, then one can construct a subgroup of large degree in its $p$-torsion, and that this construction is canonical. Another possible approach is to assume the existence of a subgroup of large degree, and then to prove that it is canonical and verifies some other properties. More precisely, one can prove the following theorem by simple arguments.

\begin{theo} \label{alternative}
Let $G_0$ be a $p$-divisible group of height $h_0$ and dimension $d_0$ over $O_K$. Suppose that there exists a finite flat subgroup $C_0 \subset G_0[p]$ of height $d_0$ such that $\deg C_0^D < 1/2$. Then $C_0$ is the unique finite flat subgroup of $G_0[p]$ satisfying these properties. Moreover, we have the relation $\deg C_0^D = ha(G_0)$, and $C_0$ is the kernel of the Frobenius in $G \times_{O_K} O_{K,\{1 - ha(G_0)\}}$.
\end{theo}

\begin{proof}
Let $H$ be a finite flat subgroup of $G[p]$ of height $d_0$ and let $h$ be the height of $H \cap C_0$. Suppose that $h \leq d_0 - 1$, and that $\deg H > d_0 - 1/2$. Then we have by the properties of the degree function (\cite{Fa_deg})
$$\deg H + \deg C_0 \leq \deg (H+C_0) + \deg (H \cap C_0) \leq \deg G[p] + h \leq 2d_0 - 1 \text{,} $$
and we get a contradiction since both $H$ and $C_0$ have a degree strictly larger than $d_0 - 1/2$. \\
\indent Let $w = \deg C_0^D = \deg G_0[p] / C_0$. We have an exact sequence
$$ 0 \to \omega_{G_0[p]/C_0} \to \omega_{G_0,\{1\}}  \to \omega_{C_0} \to 0 \text{.} $$
We thus have an isomorphism $\omega_{G_0,\{1-w\}} \simeq \omega_{C_0,\{1-\omega\}}$. The Verschiebung map $V : \omega_{G_0,\{1-w\}} \to~\omega_{G_0,\{1-w\}}^{(p)}$ has a determinant of valuation $ha(G)$ by definition. On the other side, one can filter the subgroup $C_0$ by finite flat subgroups 
$$0=H_0 \subset H_1 \subset \dots \subset H_{d_0}=~C_0 \text{,} $$
such that $H_{i}/ H_{i-1}$ has height $1$ for all $1 \leq i \leq d_0$. We thus get a filtration
$$0 \subset \omega_{C_0/H_{d_0-1}} \subset \dots \subset \omega_{C_0 / H_1} \subset \omega_{C_0} \text{.} $$
Let $\omega_{C_0/H_i,\{1-w\}}$ denote the image of $\omega_{C_0/H_i}$ in $\omega_{C_0,\{1-w\}}$. We have exact sequences
$$0 \to \omega_{C_0/H_i,\{1-w\}} \to \omega_{C_0,\{1-w\}} \to \omega_{H_i} \otimes_{O_K} O_{K,\{1-w\}} \to 0$$
for all $1 \leq i \leq d_0-1$. We claim that $\omega_{H_i} \otimes_{O_K} O_{K,\{1-w\}}$ is a free module over $O_{K,\{1-w\}}$ of rank $i$. Indeed, $\omega_{H_i}$ is generated by $i$ elements, so there is a surjective map
$$(O_{K,\{1\}})^i \to \omega_{H_i} \text{.} $$
The kernel of this map is killed by any element of valuation greater than $i - \deg H_i = \deg H_i^D \leq w$. The previous map is thus an isomorphism after tensoring with $O_{K,\{1-w\}}$, and this proves the claim. \\
Since $\omega_{C_0,\{1-w\}}$ and $\omega_{H_i} \otimes_{O_K} O_{K,\{1-w\}}$ are free modules over $O_{K,\{1-w\}}$ of rank respectively $d_0$ and $i$, we deduce that $\omega_{C_0/H_i,\{1-w\}}$ is a free $O_{K,\{1-w\}}$-module of rank $d_0-i$. We have thus filtered the module $\omega_{C_0,\{1-w\}}$ by free $O_{K,\{1-w\}}$ subspaces, and the graded pieces are isomorphic to $\omega_{H_i/H_{i-1}} \otimes_{O_K} O_{K,\{1-w\}}$. The groups $H_i/ H_{i-1}$ satisfy the conditions of those studied by Oort-Tate (\cite{T-O}). In particular, the Verschiebung map for this group is the multiplication by an element whose valuation is the degree of the dual of the group. Putting everything together, one gets
$$ha(G) = \sum_{i=1}^{d_0} \deg (H_i/H_{i-1})^D = \deg C_0^D = w \text{,} $$
this equality being in $[0,1-w]$. Since $w < 1 - w$, we have the relation $ha(G) = \deg C_0^D$. \\
\indent For the last relation, one observes that the morphism 
$$\omega_{C_0^D} \otimes_{O_K} O_{K,\{1-w\}}  \to \omega_{G^D,\{1-w\}}$$
is $0$ since the degree of $C_0^D$ is strictly less than $w$. We can then apply Proposition $1$ in \cite{Fa}.
\end{proof}

Note that there is no assumption on $p$ in the previous theorem. The canonical subgroup constructed by Fargues is thus uniquely determined by the fact that it has height the dimension of the $p$-divisible group, and that its dual has degree strictly less than $1/2$ (or equivalently that its degree is strictly larger than its height minus $1/2$). This leads to the following definition.

\begin{defi} \label{def_can}
Let $G_0$ be a $p$-divisible group of height $h_0$ and dimension $d_0$ over $O_K$. Let $C_0$ be a finite flat subgroup of $G_0[p]$. We say that $C_0$ is the canonical subgroup of $G_0$ if the height of $C_0$ is $d_0$ and if $\deg C_0^D < 1/2$.
\end{defi}

Theorem $\ref{main_theo}$ then gives a criterion for the existence of a canonical subgroup. Note that Theorem $\ref{alternative}$ says that the existence of a canonical subgroup implies the relation $ha(G) < 1/2$.

\subsection{The partial degrees of the canonical subgroup}

We have seen that one can relate the degree of the canonical subgroup and the Hasse invariant. When one considers a $p$-divisible group $G$ with an action of $O_F$, much more can be said. We keep the notations from the previous sections. We have the following result.

\begin{theo} \label{theo_1}
Let $G$ be a $p$-divisible group over $O_K$ with an action of $O_F$ satisfying Hypothesis $\ref{hypothese}$. Suppose that there exists a canonical subgroup $C \subset G[p]$ (in the sense of Definition $\ref{def_can}$).
\begin{enumerate}
\item  Suppose that $ha(G) < \min(1/e,1/2)$; then we have
$$\deg_i C^D = ha_i(G) \text{ and } \deg_i C[\pi]^D = ha_i^{[e]}(G)$$
for all $1 \leq i \leq f$. This implies $\deg C[\pi]^D = ha^{[e]}(G)$. 
\item Under the same hypotheses, we have
$$\deg_i^{[1]} C[\pi]^D = hasse_i(G) \text{  and  } \deg_i^{[j]} C[\pi]^D = m_i^{[j]}(G)$$
for all $1 \leq i \leq f$ and $2 \leq j \leq e$. 
\item If $e=1$ we suppose that $ha(G) < 1/(p+1)$; if $e \geq 2$, we suppose that $ha(G) <1/(pe) $ as well as the existence of a canonical subgroup for $G/C$. Then $ha_i(G/C) = p \cdot ha_{i-1}(G)$ for $1 \leq i \leq f$. We also have
$$ha_i^{[1]}( G/ C[\pi]) = p \cdot ha_{i-1}^{[e]}(G)  \text{ and } ha_i^{[j]} (G/C[\pi]) = ha_i^{[j-1]} (G)$$
for all $1 \leq i \leq f$ and $2 \leq j \leq e$. The Hasse invariant of $G/C[\pi]$ is then equal to $ha(G)+~(p-~1) ha^{[e]}(G)$. Moreover, if $e \geq 2$
$$hasse_i(G/C[\pi]) = p \cdot m_{i-1}^{[e]} (G) \text{  ,  }  m_i^{[2]} (G/C[\pi]) = hasse_i(G)  \text{ and  } m_i^{[j]} (G/C[\pi]) = m_i^{[j-1]} (G)$$
for all $1 \leq i \leq f$ and $3 \leq j \leq e$. 
\end{enumerate}
\end{theo}

From the result of Fargues (\cite{Fa}), the existence of the canonical subgroup for $G$ is guaranteed by the conditions $p >3$ and $ha(G) < 1/2$, or $p=3$ and $ha(G) < 1/3$. The existence of the canonical subgroup for $G/C$ is guaranteed by the conditions $p>3$ and $ha(G) < 1/(2p)$, or $p=3$ and $ha(G) < 1/(3p)$. \\ 
\indent One can then not only relate the degree of $C$ to the Hasse invariant, but also the degree and partial degrees of $C[\pi]$ to the partial Hasse invariants. One can also compute the partial Hasse invariants of the $p$-divisible group $G/C[\pi]$. Actually, one can get information on $C[\pi^k]$. Indeed, we have the following propositions.

\begin{prop} \label{prop_1}
Let $G$ be a $p$-divisible group over $O_K$ with an action of $O_F$ satisfying Hypothesis $\ref{hypothese}$, and suppose that there exists a canonical subgroup $C \subset G[p]$. Suppose that $e \geq 2$ and let $1 \leq k \leq e$ be an integer. Suppose that $ha(G) < 1/e$; then we have
$$\deg_i (C[\pi^k] / C[\pi^{k-1}])^D =  ha_i^{[e+1-k]}(G)$$
for all $1 \leq i \leq f$. Thus $\deg (C[\pi^k]/C[\pi^{k-1}])^D =  ha^{[e+1-k]}(G)$. Suppose that $ha(G) < 1/(pe)$, and that there exists a canonical subgroup for $G/C$. Then for all $1 \leq i \leq f$ we have
\begin{itemize}
\item $\deg_i^{[j]} (C[\pi^k]/C[\pi^{k-1}])^D = p \cdot m_{i-1}^{[e+1-k+j]}(G)$ for $1 \leq j \leq k-1$.
\item $\deg_i^{[k]} (C[\pi^k]/C[\pi^{k-1}])^D = hasse_i(G)$.
\item $\deg_i^{[j]} (C[\pi^k]/C[\pi^{k-1}])^D = m_i^{[j-k+1]}(G)$ for $k+1 \leq j \leq e$.
\end{itemize}
\end{prop}

\begin{prop} \label{prop_2}
Let $G$ be a $p$-divisible group over $O_K$ with an action of $O_F$ satisfying Hypothesis $\ref{hypothese}$, and suppose that there exists a canonical subgroup $C \subset G[p]$. Suppose $e \geq 2$, $ha(G) < 1/(pe)$ and that there exists a canonical subgroup for $G/C$. Then we have
$$(ha_i^{[j]}( G/ C[\pi^k]))_{1 \leq j \leq e}  =( p \cdot ha_{i-1}^{[e-k+1]}(G), \dots, p \cdot ha_{i-1}^{[e]}(G), ha_i^{[1]} (G), \dots, ha_i^{[e-k]}(G))$$
for $1 \leq i \leq f$ and $1 \leq k \leq e$. These relations are equivalent to
\begin{equation*}
\begin{aligned}
& (hasse_i(G/C[\pi^k]), m_i^{[2]}(G/C[\pi^k]), \dots, m_i^{[e]}(G/C[\pi^k])) =\\
& (p \cdot m_{i-1}^{[e-k+1]}(G), \dots, p \cdot m_{i-1}^{[e]}(G), hasse_i(G), m_i^{[2]}(G), \dots, m_i^{[e-k]}(G))
\end{aligned}
\end{equation*}
for all $1 \leq i \leq f$ and $1 \leq k \leq e-1$. For $k=e$ we get $hasse_i(G/C) = p \cdot hasse_{i-1}(G)$ and $m_i^{[j]}(G/C) = p \cdot m_{i-1}^{[j]}(G)$ for all $1 \leq i \leq f$ and $2 \leq j \leq e$.
\end{prop}

We could have presented these two results as corollaries of the theorem, but to get sharper bounds on the Hasse invariant, we will prove these three results together. This will be done in the next section. We sum up all the information in the following two tables. \\

\newpage

\begin{table}[!h] 
\begin{displaymath}
\begin{array}{|c | c | c | c | c | c |}
\hline
              & C[\pi]^D & (C[\pi^2]/C[\pi])^D      & \dots & (C/C[\pi^{e-1}])^D & C^D \\ 
\hline
\deg_i^{[1]}  &  hasse_i(G) & p \cdot m_{i-1}^{[e]}(G)  &  \dots & p \cdot m_{i-1}^{[2]}(G) & ha_i^{[1]}(G) \\
\hline
\deg_i^{[2]}   &  m_i^{[2]}(G) & hasse_i(G)          & \dots    &  p \cdot m_{i-1}^{[3]} (G) &  ha_i^{[2]}(G) \\
\hline
\vdots & \vdots & \vdots & \ddots & \vdots & \vdots  \\
\hline
\deg_i^{[e]}  &   m_i^{[e]}(G) & m_i^{[e-1]}(G) & \dots & hasse_i(G) & ha_i^{[e]}(G) \\
\hline
\deg_i  &   ha_i^{[e]}(G) & ha_i^{[e-1]}(G) & \dots & ha_i^{[1]}(G) & ha_i(G) \\
\hline
\deg  &   ha^{[e]}(G) & ha^{[e-1]}(G) & \dots & ha^{[1]}(G) & ha(G) \\
\hline
\end{array}
\end{displaymath}
\caption{The partial degrees of the graded parts of the canonical subgroup} \label{tab}
\end{table}

\begin{table}[!h]
\begin{displaymath} 
\begin{array}{| c | c | c | c | c |}
\hline
         & G/C[\pi] & G/C[\pi^2] & \dots & G/C \\
\hline
hasse_i  &  p \cdot m_{i-1}^{[e]}(G) & p \cdot m_{i-1}^{[e-1]}(G) & \dots & p \cdot hasse_{i-1}(G) \\
\hline
m_i^{[2]} & hasse_i(G) & p \cdot m_{i-1}^{[e]}(G)  & \dots & p \cdot m_{i-1}^{[2]}(G) \\
\hline
m_i^{[3]} & m_i^{[2]}(G) & hasse_i(G) & \dots & p \cdot m_{i-1}^{[3]}(G) \\
\hline
\vdots  & \vdots & \vdots & \ddots & \vdots  \\
\hline
m_i^{[e]}  & m_i^{[e-1]}(G) & m_i^{[e-2]}(G) & \dots & p \cdot m_{i-1}^{[e]}(G) \\
\hline
\end{array}
\end{displaymath}
\caption{The primitive Hasse invariants of the $p$-divisible groups $G/C[\pi^k]$} \label{tabl}
\end{table}

These two tables are valid if $G$ is a $p$-divisible group as in Theorem $\ref{theo_1}$, $C$ is the canonical subgroup, $ha(G) < \min(1/(pe),1/(p+1))$ and if there exists a canonical subgroup for $G/C$.

\begin{rema}
It is not difficult to see that if $G$ does not satisfy Hypothesis $\ref{hypothese}$, then $ha(G)=1$ and there cannot be a canonical subgroup for $G[p]$. We could thus have removed this assumption in the theorem and in the propositions.
\end{rema}

\begin{rema}
The condition $ha(G) < 1/e$ imply that all the partial and primitive Hasse invariants of $G$ are well defined. If $ha(G) < 1/e$ and $C$ is the canonical subgroup of $G[p]$, then we have
$$\deg_i C[\pi]^D + r_i(G) \leq ha_i^{[e]}(G) + \sum_{j=1}^{e-1} ha_{i+1}^{[j]} (G) \leq ha(G) < 1/e \text{.} $$
Thus the partial degrees of $C[\pi]^D$ are well defined too thanks to Proposition $\ref{indep_deg}$.
\end{rema}

\begin{rema}
If $G$ satisfies the Rapoport condition, then we have $m_i^{[j]}(G)=0$ for all $1 \leq i \leq f$ and $2 \leq j \leq e$. Then we have 
$$\deg_i (C[\pi^k] / C[\pi^{k-1}])^D = hasse_i(G)$$
for all $1 \leq i \leq f$ and $1 \leq k \leq e$. Moreover, the element $\deg_i^{[j]} (C[\pi^k] / C[\pi^{k-1}])^D$ is equal to $hasse_i(G)$ if $j=k$ and $0$ otherwise. Let $k$ be an integer between $1$ and $e-1$. The elements $(hasse_i(G / C[\pi^{k}]), m_i^{[j]}(G/C[\pi^k]))$ are all $0$ except the relation $m_i^{[k+1]}(G/C[\pi^k]) = hasse_i(G)$ for all $1 \leq i \leq f$. Note that if $G$ is not ordinary, then $G/C[\pi^k]$ does not satisfy the Rapoport condition for $1 \leq k \leq e-1$. This is consistent with the fact that the $U_{\pi}$ operator on the Hilbert modular variety does not stabilize the Rapoport locus (see \cite{A-G}).
\end{rema}

\begin{rema}
Suppose that $e \geq 2$, that the Hasse invariant of $G$ is small enough, and that there exists a canonical subgroup $C \subset G[p]$. Consider $G'=G/C[\pi]$, it has a subgroup $H_0'=G[\pi]/C[\pi]$. This subgroup has partial degrees 
$$(\deg_i^{[j]} H_0')_{1 \leq j \leq e} = (hasse_i(G), m_i^{[2]}(G), \dots, m_i^{[e]}(G))$$
for all $1 \leq i \leq f$. In $G'' = G / C[\pi^2]$, the image of $H_0'$ is $H_0'' = G'[\pi] / (C[\pi^{2}] / C[\pi])$. It has partial degrees 
$$(\deg_i^{[j]} H_0'')_{1 \leq j \leq e} = (p \cdot m_{i-1}^{[e]}(G), hasse_i(G), m_i^{[2]}(G), \dots, m_i^{[e-1]}(G))$$
for all $1 \leq i \leq f$. This shows that the $U_{\pi}$ operator on the Hilbert modular variety does not increase in general any of the partial degrees.
\end{rema}

\subsection{Proof of the theorem} 

Before proving the theorem, we recall the structure theorem of Raynaud (\cite{Ra}) concerning finite flat group schemes of height $f$ over $O_K$, of $p$-torsion and with an action of $O_{F^{ur}}$.

\begin{prop}
Let $H$ be a finite flat group scheme of height $f$ over $O_K$, of $p$-torsion and with an action of $O_{F^{ur}}$. Then there exists elements $(a_i,b_i)_{1 \leq i \leq f}$ of $O_K$ such that $a_i b_i = p u$ (where $u$ is a fixed $p$-adic unit), with $H$ isomorphic to the spectrum of 
$$O_K[X_1, \dots, X_f] / (X_{i}^p - a_{i+1} X_{i+1}) \text{,} $$
where we identify $X_{f+1}$ and $X_1$. The dual of the group with parameters $(a_i,b_i)$ is the one with parameter $(b_i,a_i)$. Moreover, we have $\omega_{H,i} = O_K / a_i$, and therefore $\deg_i H = v(a_i)$, $\deg_i H^D = v(b_i)$ for all $1 \leq i \leq f$. The Verschiebung map $\omega_{H,i} \to \omega_{H^{(p)},i-1}$ is given by
\begin{equation*}
\begin{aligned}
O_K/ a_i & \to   O_K / (p,a_{i-1}^p) \\
1 & \to  b_i
\end{aligned}
\end{equation*}
for all $1 \leq i \leq f$.
\end{prop}

We will refer to such group schemes as Raynaud group schemes.

\begin{rema}
Since the group is defined over $O_K$, the condition $(\star \star)$ of \cite{Ra} is automatically satisfied.
\end{rema}

We now turn to the proof of Theorem $\ref{theo_1}$ and of Propositions $\ref{prop_1}$ and $\ref{prop_2}$. Let $G$ and $C$ be as in the theorem. Let us write $w=ha(G)$. We have an exact sequence
$$0 \to \omega_{G[p]/C} \to \omega_{G,\{1\}} \to \omega_C \to 0 \text{.} $$
But the degree of $G[p] / C$ is $w$, thus we have an isomorphism
$$\omega_{G,\{1 -w\}} \simeq \omega_{C,\{1-w\}} \text{.} $$
It gives isomorphisms $\omega_{G,i,\{1 -w\}} \simeq \omega_{C,i,\{1-w\}}$ for all $1 \leq i \leq f$. The space $C(O_K)$ is a $\mathbb{F}_q$-vector space, with a map $[\pi]$ such that $[\pi]^e = 0$. Let us give a filtration on this vector space; it gives a filtration $H_1 \subset \dots \subset  H_{de} = C$, such that each $H_{k+1} / H_k$ is a Raynaud group scheme. Moreover, one can do it in a way that $H_{kd} = C[\pi^k]$ for all $1 \leq k \leq e$. This gives a filtration
$$0 \subset \omega_{C/H_{de-1}} \subset \dots \subset \omega_{C/H_1} \subset \omega_C \text{.} $$
Let $\omega_{C/H_k,i,\{1 - w\}}$ be the image of $\omega_{C/H_k}$ in $\omega_{C,i,\{1 - w\}}$. We have thus filtered the free $O_{K,\{1-w\}}$-module of rank $de$ $\omega_{C,i,\{1-w\}}$ by $de$ submodules, such that each of the graded pieces is monogenous (since $H_{k+1}/H_k$ is a Raynaud group scheme, the quotient $\omega_{C/H_k,i} / \omega_{C/H_{k+1},i} \simeq \omega_{H_{k+1}/H_k,i}$ is monogenous). This forces each module $\omega_{C/H_k,i,\{1 - w\}}$ to be free of rank $de-k$ over $O_{K,\{1-w\}}$. We have filtered $\omega_{C,i,\{1-w\}}$ by free $O_{K,\{1-w\}}$ submodule, and the Verschiebung acts on the graded pieces by the multiplication by an element of valuation $\deg_i (H_k / H_{k-1})^D$.  \\
\indent We will first compute the unramified partial degrees of $C$. When one takes the valuation of the determinant of the Verschiebung acting on $\omega_{C,i,\{1-w\}}$, one gets the following equality in $[0,1-w]$
$$ha_i(G) = \sum_{k=1}^{de} \deg_i (H_k / H_{k-1})^D= \deg_i C^D$$
for $1 \leq i \leq f$. Since $ha_i(G) \leq w < 1-w$, this relation is simply an equality. This settles the first assertion of Theorem $\ref{theo_1}$ in the case $e=1$. \\
\indent We now assume $e \geq 2$. The filtration
$$0 \subset \omega_{C/C[\pi^{e-1}],i,\{1/e\}} \subset \dots \subset \omega_{C/C[\pi],i,\{1/e\}} \subset \omega_{C,i,\{1/e\}}$$
is thus an adequate filtration of $\omega_{G,i,\{1/e\}} \simeq \omega_{C,i,\{1/e\}}$ for all $1 \leq i \leq f$. Moreover, from the result on Raynaud group schemes, the determinant of the map
$$V : \omega_{C/C[\pi^{k-1}],i,\{1/e\}} / \omega_{C/C[\pi^{k}],i,\{1/e\}} \to (\omega_{C/C[\pi^{k-1}],i-1,\{1/e\}} / \omega_{C/C[\pi^{k}],i-1,\{1/e\}})^{(p)}$$
has a determinant with valuation equal to $\deg_i (C[\pi^k] / C[\pi^{k-1}])^D$ for $1 \leq i \leq f$ and $1 \leq k \leq e$. From Propositions $\ref{indep_h}$ and $\ref{compute_hasse}$, one gets
$$ha_i^{[e+1-k]}(G) = \deg_i (C[\pi^k] / C[\pi^{k-1}])^D$$
for all $1 \leq i \leq f$ and $1 \leq k \leq e$. This concludes the proof of the first assertion of Theorem $\ref{theo_1}$, and the first part of Proposition $\ref{prop_1}$. \\
$ $\\
\indent Now let us turn to the computation of the partial Hasse invariants of the $p$-divisible groups $G/C[\pi^k]$, for $1 \leq k \leq e$. The $p$-divisible group $(G/C) \times_{O_K} O_{K,\{1-w\}}$ is isomorphic to $(G \times_{O_K} O_{K,\{1-w\}})^{(p)}$, where the subscript means a twist by the Frobenius. First we suppose that and $w < 1/(p+1)$. We have the following equality in $[0,1-w]$
$$ha_i(G/C) = p \cdot ha_{i-1} ( G)$$
for $1 \leq i \leq f$. Since $p \cdot ha_{i-1} (G) \leq pw < 1-w$, this is just an equality. This proves the first assertion of Theorem $\ref{theo_1}$ in the case $e=1$. \\
\indent Now we suppose that $e \geq 2$ and $w < 1/(pe)$. Since the element $p \cdot w$ is strictly less than $1/e$, we have
$$ha_i^{[j]} (G/C) = p \cdot ha_{i-1}^{[j]} (G)$$
for all $1 \leq i \leq f$ and $1 \leq j \leq e$. This gives the result for $G/C$. Now assume the existence of a canonical subgroup for $G/C$. We will write $C_2 \subset G[p^2]$ the subgroup such that $C_2/C$ is the canonical subgroup of $G/C$. Since $ha(G/C) < 1/e$, we can apply our previous result to this $p$-divisible group. We note that $(C_2/C)[\pi^k] = C_2[\pi^{k+e}]/C$, and get
$$\deg_i (C_2 [\pi^{e+k}]/ C_2[\pi^{e+k-1}])^D = p \cdot ha_{i-1}^{[e+1-k]}(G)$$
for all $1 \leq i \leq f$ and $1 \leq k \leq e$. Next, we consider the $p$-divisible group $G/C[\pi^k]$. One easily checks that it as a canonical subgroup equal to $C_2[\pi^{k+e}] / C[\pi^k]$. Applying our previous result to this $p$-divisible group, one gets
$$ha_i^{[j]} ( G/C[\pi^k]) = \deg_i (C_2[\pi^{k+e+1-j}] / C_2[\pi^{k+e-j}])^D$$
for all $1 \leq i \leq f$, $1 \leq j \leq e$ and $1 \leq k \leq e$. Putting all these relations together, we conclude that
$$(ha_i^{[j]} ( G/C[\pi^k]))_{1 \leq j \leq e} = ( p \cdot ha_{i-1}^{[e+1-k]}(G), \dots, p \cdot ha_{i-1}^{[e]}(G), ha_i^{[1]}(G), \dots, ha_i^{[e-k]}(G))$$  
for $1 \leq i \leq f$ and $1 \leq k \leq e$. This proves the third assertion of Theorem $\ref{theo_1}$ and Proposition $\ref{prop_2}$. \\
$ $\\
\indent We will now compute the partial degrees of $C[\pi]^D$. We will prove that $\deg_i^{[j]} C[\pi]^D = m_i^{[j]}(G)$ for all $1 \leq i \leq f$ and $2 \leq j \leq e$. Since we have already computed the unramified partial degree $\deg_i C[\pi]^D = ha_i^{[e]}(G)$, this will imply that $\deg_i^{[1]} C[\pi]^D = hasse_i(G)$. \\
Let us fix an integer $i$ between $1$ and $f$, and let us denote by $\nu_{C[\pi]^D,i}$ the cokernel of the map $\omega_{(G/C[\pi])^D,i}^\vee \to \omega_{G^D,i}^\vee$. We thus have an exact sequence
$$0 \to \omega_{(G/C[\pi])^D,i}^\vee \to \omega_{G^D,i}^\vee \to \nu_{C[\pi]^D,i} \to 0 \text{.} $$
Since $\deg_i C[\pi]^D = ha_i^{[e]}(G) < 1/e$, we have an exact sequence
$$\omega_{(G/C[\pi])^D,i,\{1/e\}}^\vee \to \omega_{G^D,i,\{1/e\}}^\vee \to \nu_{C[\pi]^D,i} \to 0 \text{.} $$
On $\omega_{G^D,i,\{1/e\}}^\vee$, the action of $[\pi]$ is given by a matrix of the form
\begin{displaymath}
\left(
\begin{array}{cccc}
0 & M_i^{[e]} & \dots & * \\
& 0 & \ddots & * \\
& & \ddots  & M_i^{[2]} \\
& & & 0 
\end{array}
\right) \text{,} 
\end{displaymath}
where this matrix is written in a basis respecting the filtration $(\omega_{G^D,i,\{1/e\}} / \omega_{G^D,i,\{1/e\}}^{[j]})^\vee$, and the determinant of the matrix $M_i^{[j]}$ has valuation $m_i^{[j]}(G)$ for all $2 \leq j \leq e$. Here, all the blocks are of size $h-d$. By a slight abuse of notation, we will still denote by $\omega_{(G/C[\pi])^D,i,\{1/e\}}^\vee$ the image of this module in $\omega_{G^D,i,\{1/e\}}^\vee$, and will work with this module from now on. This module contains $[\pi] \cdot \omega_{G^D,i,\{1/e\}}^\vee$, and the image of this module in the quotient $\omega_{G^D,i,\{1/e\}}^\vee /  [\pi] \cdot \omega_{G^D,i,\{1/e\}}^\vee$ is generated by $h-d$ elements (since the height of $G[\pi]/C[\pi]$ is $f(h-d)$). We will write the matrix of these $h-d$ elements by
\begin{displaymath}
Y = \left(
\begin{array}{c}
Y_1 \\
\vdots \\
Y_e
\end{array}
\right)  \text{.} 
\end{displaymath}
Since the module $\nu_{C[\pi]^D,i}$ contains as a quotient a module isomorphic to $O_{K,\{1/e\}}^{h-d} / Y_e O_{K,\{1/e\}}^{h-d}$, we see that $v(\det Y_e) \leq \deg_i C[\pi]^D =ha_i^{[e]}(G)$. We now claim that the intersection of $\omega_{(G/C[\pi])^D,i,\{1/e\}}^\vee$ with the first step of the filtration $(\omega_{G^D,i,\{1/e\}} / \omega_{G^D,i,\{1/e\}}^{[e-1]})^\vee$ is generated by the image of the matrix $M_i^{[e]}$. Indeed, let us write by $V_2, \dots, V_e$ the (non zero) columns of the matrix of $[\pi]$. An element $X$ in $\omega_{(G/C[\pi])^D,i,\{1/e\}}^\vee$ can then be written as a linear combination
$$X = V_2 \alpha_1 + \dots + V_e \alpha_{e-1} + Y \alpha_e \text{,} $$
for some $(h-d) \times 1$ columns $\alpha_i$. If $X$ is in the first step of the filtration, then one sees that $Y_e \alpha_e = 0$ in $O_{K,\{1/e\}}^d$. This implies that the elements of $\alpha_e$ have a valuation greater than $1/e - ha_i^{[e]}(G)$. We then get the relation $M_i^{[2]} \alpha_{e-1} = 0 $ in $O_{K,\{1/e - ha_i^{[e]}(G)\}}^d$. The valuations of the elements of $\alpha_{e-1}$ are thus greater than $1/e - ha_i^{[e]}(G) - m_i^{[2]}(G)$. By induction, one sees that the coefficients of $\alpha_2, \dots, \alpha_e$ have all a valuation greater than $1/e - ha_i^{[e]}(G) - m_i^{[2]}(G) - \dots - m_i^{[e-1]}(G)$. This concludes the claim, with the hypothesis
$$ha_i^{[e]}(G) + \sum_{j=2}^{e} m_i^{[j]}(G) < 1/e \text{.} $$
This hypothesis guarantees that $\deg_i^{[e]} C[\pi]^D = m_i^{[e]}(G)$. Reasoning by induction, considering the module $(\omega_{G^D,i,\{1/e\}}^{[e-1]})^\vee$, one gets that under the same hypothesis 
$$\deg_i^{[j]} C[\pi]^D = m_i^{[j]}(G)$$
for all $2 \leq j \leq e$. This concludes the second assertion of Theorem $\ref{theo_1}$ since the hypothesis is implied by the relation $ha(G) < 1/e$. Indeed, if $e \geq 2$, one has
$$ha_i^{[e]}(G) + \sum_{j=2}^{e} m_i^{[j]}(G) \leq 2 ha_i^{[e]}(G) \leq 2 ha^{[e]}(G) \leq ha^{[e]}(G) + ha^{[e-1]}(G) \leq ha(G)$$
for all $1 \leq i \leq f$. \\
$ $\\
\indent To conclude, it remains to prove the second part of Proposition $\ref{prop_1}$, i.e. to compute the partial degrees of $(C[\pi^{k+1}]/C[\pi^{k}])^D$ for $1 \leq k \leq e-1$ with the assumption that $ha(G) < 1/(pe)$ and the existence of a canonical subgroup for $G/C$. We want to apply our previous result to the $p$-divisible group $G/C[\pi^k]$. For this we need the hypothesis
$$ha_i^{[e]}(G/C[\pi^k]) + \sum_{j=2}^{e} m_i^{[j]}(G/C[\pi^k]) < 1/e$$
for all $1 \leq i \leq f$ and $1 \leq k \leq e-1$. But from the computation on the primitive Hasse invariants of $G/C[\pi^k]$, we have
\begin{equation*}
\begin{aligned}
ha_i^{[e]}(G/C[\pi^k]) + \sum_{j=2}^{e} m_i^{[j]}(G/C[\pi^k]) & = ha_i^{[e-k]}(G) + p \sum_{j=e-k+2}^e m_{i-1}^{[j]} (G) + hasse_i(G) + \sum_{j=2}^{e-k} m_i^{[j]}(G)  \\
& \leq 2 ha_i^{[e-k]}(G) <2/(pe) \leq 1/e 
\end{aligned}
\end{equation*}
for all $1 \leq i \leq f$ and $1 \leq k \leq e-1$. This concludes the proof of Theorem $\ref{theo_1}$ and Propositions $\ref{prop_1}$ and $\ref{prop_2}$.

\begin{rema}
Incidentally, we have proved that if $H$ is a finite flat subgroup of $G[\pi]$ of height $f(h-d)$ such that 
$$\deg_i H + \sum_{j=2}^e m_i^{[j]}(G) < 1/e$$
for some integer $i$, then $\deg_i^{[j]} H = m_i^{[j+1]}(G)$ for all $1 \leq j \leq e-1$.
\end{rema}

\bibliographystyle{amsalpha}

\end{document}